\documentclass{amsart}
\usepackage{array}
\usepackage{xcolor}
\usepackage{moreverb}
\usepackage{amssymb,amsmath,amsthm}
\usepackage{mathtools}
\usepackage{enumitem}
\usepackage{tikz}
\usepackage{tikz-cd}
\usepackage{graphicx}
\usepackage{ulem}
\usepackage{url}
\usepackage{hyperref}

\usepackage[english]{babel}
\usepackage[latin1]{inputenc}
\usepackage{times}
\usepackage[T1]{fontenc}
\usepackage{listings}

\makeatletter
\newcommand{\newsectionstyle}{%
	\renewcommand{\@secnumfont}{\bfseries}
	\renewcommand\section{\@startsection{section}{2}%
	  \z@{.5\linespacing\@plus.7\linespacing}{-.5em}%
	  {\normalfont\bfseries}}%
}
\let\oldsection\section%
\let\old@secnumfont\@secnumfont%
\newcommand{\originalsectionstyle}{%
\let\@secnumfont\old@secnumfont
\let\section\oldsection
}
\makeatother

\numberwithin{equation}{section}

 \newcommand{\RR}{\mathbb{R}}
\newcommand{\CC}{\mathbb{C}} 
 
\newcommand{\NN}{\mathbb{N}}

\newcommand{\rplus}{\mathbb{R}_{>0}}
\newcommand{\disp}{\displaystyle}
\DeclarePairedDelimiter{\bracepair}{\lbrace}{\rbrace}

\DeclarePairedDelimiter{\parenpair}{(}{)}
\DeclarePairedDelimiter{\vertpair}{\vert}{\vert}
\DeclarePairedDelimiter{\Vertpair}{\Vert}{\Vert}

\newcommand{\abs}[1]{\vertpair*{#1}}

\newcommand{\card}[1]{\vertpair{#1}}
\newcommand{\supnorm}[1]{\Vertpair*{#1}_{\infty}}

\newcommand{\bd}[1]{\partial #1}

\renewcommand{\Re}{\operatorname{Re}}
\renewcommand{\Im}{\operatorname{Im}}

\newcommand{\metricball}[2]{\mathbb{D}_{#2}\left(#1\right)}

\newcommand{\declare}{\equiv}

\DeclareMathOperator{\sgn}{sign}

\newtheorem{thm}{Theorem}[section]
\newtheorem{cor}[thm]{Corollary}
\newtheorem{lem}[thm]{Lemma}
\newtheorem{prop}[thm]{Proposition}

{\theoremstyle{definition}  }
{\newtheorem{rem}[thm]{Remark}}
{\theoremstyle{definition} \newtheorem{smallrem}[thm]{Remark}}
{\theoremstyle{remark} }

\newtheorem{bigclm}[thm]{Claim}
{\theoremstyle{remark} }

\newcommand{\cvec}{\mathbf{c}}
\newcommand{\cveclong}{(c_1, c_2, \dotsc, c_n)}
\newcommand{\cvecgeom}{\gamma}
\newcommand{\cvecmin}{D_*}
\newcommand{\cvecmax}{D^*}

\newcommand{\dvec}{\mathbf{d}}
\newcommand{\dshort}[3]{\bracepair*{#1, #2; #3}}
\newcommand{\dlong}[3]{\bracepair*{\left( #1, #2 \right)}_{j = 0}^{#3}}
\newcommand{\dgroupideal}{\mathbf{d^*}}

\newcommand{\specw}{\widetilde{w}} 

\newcommand{\cspace}[3]{C(#1; #2, #3)}
\newcommand{\dspace}[3]{D(#1; #2, #3)}
\newcommand{\countnroots}[1]{\nu_- \parenpair*{#1}}
\newcommand{\countzroots}[1]{\nu_0 \parenpair*{#1}}
\newcommand{\countproots}[1]{\nu_+ \parenpair*{#1}}
\newcommand{\countpzroots}[1]{\overline{\nu} \parenpair*{#1}}
\newcommand{\countnrootstime}[1]{\omega_- \parenpair*{#1}}
\newcommand{\countzrootstime}[1]{\omega_0 \parenpair*{#1}}
\newcommand{\countprootstime}[1]{\omega_+ \parenpair*{#1}}
\newcommand{\countpzrootstime}[1]{\overline{\omega} \parenpair*{#1}}
\newcommand{\cvecideal}{\mathbf{c}^*}

\newcommand{\abaverage}{\widetilde{d}}

\newcommand{\permmatspecial}{\Sigma_*}
\newcommand{\permmat}{\Sigma}
\newcommand{\sympoly}[2]{\sigma_{#1}(#2)}

\newcommand{\wall}{\mathop{Wall}}
\newcommand{\ourbox}{\mathop{Box}}

\newcommand{\extplane}{E_{\text{ext}}}
\newcommand{\diver}{diversity}
\newcommand{\cvecext}{\mathbf{c}_{\text{ext}}}

\newcommand{\setn}{\overline{n}}

\newcommand{\rdisk}[1]{\mathbb{D}_{#1}}

\DeclareMathOperator{\diag}{diag}

\newcommand{\ann}[2]{\mathop{Ann}\left( #1, #2 \right)}
\newcommand{\wtder}{\dot{w}}
\newcommand{\wttder}{\dot{\widetilde{w}}}
\newcommand{\atder}{\dot{d_0}}
\newcommand{\btder}{\dot{d_q}}

\newcommand{\smallterm}{\Delta}

\title[Localization of eigenvalues]{Localization of eigenvalues of Doubly Cyclic Matrices}

\author{Charles E.~Baker}
\address{The Ohio State University, Columbus, Ohio, USA}
\email{\url{baker.1656@osu.edu}}

\author{Boris S.~Mityagin}
\address{The Ohio State University, Columbus, Ohio, USA}
\email{\url{mityagin.1@osu.edu} \\
\url{boris.mityagin@gmail.com}}

\date{\today}


\begin{document}
\begin{abstract}
For a family of doubly cyclic matrices of the form $\eqref{eq:dczplusex}$, a maximum for the number of eigenvalues in the left half-plane is attained by $X_* \in \eqref{eq:idealmatrix}$, with $\alpha,\beta \in \eqref{eq:detdczp}$.  This confirms a conjecture of C. Johnson, Z. Price, and I. Spitkovsky.

Moreover, the complete range of possibilities for the number of eigenvalues in the left half-plane is demonstrated: if $\alpha < \beta$, then any odd number between $1$ and the maximum, inclusive, is attainable.

\end{abstract}
\maketitle
\normalem
\everymath{\displaystyle}

\newsectionstyle

\section{Introduction}
For $n \in \NN$, $n \geq 2$, we consider matrices $X \in M_n(\RR)$ of a particular form.  Defining $\rplus = (0, \infty)$, and fixing vectors $\mathbf{a} = (a_1, \dotsc, a_n)$ and $\mathbf{b} = (b_1, \dotsc, b_n)$ in $(\rplus)^n$, we study the matrix
\begin{equation} \label{eq:dczplusex}
 X = \begin{pmatrix}
   a_1 &  - b_1 &     0  & \cdots  & 0\\
     0 &    a_2 &  - b_2 & \cdots  & 0\\
\vdots &        & \ddots & \ddots  & \\
     0 & \cdots & \cdots & a_{n-1} & -b_{n-1}\\
  -b_n &      0 & \cdots &       0 & a_n
\end{pmatrix} \, .
\end{equation} 
Since
\begin{equation} \label{eq:detdczp}
\det X  = \alpha^n - \beta^n, \quad \alpha \declare \left( \prod_{k = 1}^n a_k \right)^{1/n} \, \, , \quad  \beta \declare  \left( \prod_{k = 1}^n b_k \right)^{1/n}
\end{equation}
the geometric means of the $a_k$'s and $b_k$'s play a key role.  We let $DC(\alpha, \beta)$ denote the set of matrices of the form \eqref{eq:dczplusex} with given geometric mean $\alpha$ for the $a_k$'s and $\beta$ for the $b_k$'s.  

Inspired by the occurrence of such matrices in the previous paper \cite{JeffetalGenetics}, C. Johnson, Z. Price, and I. Spitkovsky, in \cite{JPS2013Distr}, consider the number of eigenvalues of such a matrix in the left half-plane.  In particular, they note that for several cases (when $n \leq 4$, or $\cos\left(\frac{2\pi}{n} \right) < \frac{\alpha}{\beta} < 1$), the number of eigenvalues in the left-half-plane is the same as that for $\alpha I - \beta \permmatspecial$.  Here, $I$ is the identity $n \times n$ matrix and 
\begin{equation} \label{eq:pdef}
 \permmatspecial = \begin{pmatrix}
   0   &      1 &     0  & \cdots  & 0\\
     0 &      0 &     1  & \cdots  & 0\\
\vdots &        & \ddots & \ddots  & \\
     0 & \cdots & \cdots &      0  & 1\\
     1 &      0 & \cdots &      0  & 0
\end{pmatrix}
\end{equation} 
is the relevant permutation $n \times n$ matrix.

Numerical evidence presented in \cite{JPS2013Distr} suggests that in general, the number of eigenvalues in the left half-plane for any matrix in $DC(\alpha, \beta)$ is bounded above by the corresponding value for $\alpha I - \beta \permmatspecial$.  In this paper, we prove this conjecture.

\begin{thm} \label{thm:eigencounts}
Fix $n \in \NN$, $n \geq 2$.  Fix $\alpha, \beta \in \rplus$.  Let $X \in DC(\alpha, \beta)$.  Then the number of eigenvalues of $X$ with negative real part does not exceed the number of eigenvalues of $\alpha I - \beta \permmatspecial$ with negative real part, and setting $X = \alpha I - \beta \permmatspecial \in DC(\alpha, \beta)$ allows us to attain this upper bound as a maximum among all elements of $DC(\alpha, \beta)$.  
\end{thm}
See also Remark~\ref{rem:wlogrem} for an extension of the claim of this theorem.

Conjugating with nonsingular matrices preserves the spectrum.  We conjugate $X$ with the diagonal matrix 
\begin{equation} \label{eq:qdef}
 Q = \begin{pmatrix}
1 & 0 & & \dotsc & 0\\
0 & \disp \frac{\beta}{b_1} & 0 & \dotsc & 0\\
0 & 0 & \frac{\beta}{b_1} \cdot \frac{\beta}{b_2} & \dotsc & 0\\
\vdots & \vdots & \vdots & \ddots & 0\\
0 & 0 & \cdots & 0 & \disp \prod_{k = 1}^{n-1} \frac{\beta}{b_k}
\end{pmatrix} \, .
\end{equation} 
Notice that $1 = \prod_{k = 1}^n \frac{\beta}{b_k}$.  Then with  $c_k = \frac{a_k}{\beta}$, $1 \leq k \leq n$,
\begin{equation} \label{eq:xtransform} 
\begin{split}
Q^{-1} X Q = \begin{pmatrix}
   a_1 & -\beta &     0  & \cdots  & 0\\
     0 &    a_2 & -\beta & \cdots  & 0\\
\vdots &        & \ddots & \ddots  & \\
     0 & \cdots & \cdots & a_{n-1} & -\beta\\
-\beta &      0 & \cdots &       0 & a_n
\end{pmatrix} &= \beta  \begin{pmatrix}
   c_1 &     -1 &     0  & \cdots  & 0\\
     0 &    c_2 &     -1 & \cdots  & 0\\
\vdots &        & \ddots & \ddots  & \\
     0 & \cdots & \cdots & c_{n-1} & -1\\
    -1 &      0 & \cdots &       0 & c_n
\end{pmatrix}  \\
&\declare \beta \widetilde{X}.
\end{split}
\end{equation}  
The factor $\beta > 0$ rescales the spectrum, but does not change the signs of the real parts of the points of the spectrum; therefore, we study $\widetilde{X}$. We have simplified the parameter scheme to an $n$-parameter system $\cvec = \cveclong \in (\rplus)^n$, with geometric mean 
\begin{equation} \label{eq:cvecgeommean}
\begin{split}
\cvecgeom & = \left( \prod_{k = 1}^n c_k \right)^{1/n} = \left( \frac{ \disp \prod_{k = 1}^n a_k}{\beta^n} \right)^{1/n} = \frac{\alpha}{\beta}.
\end{split}
\end{equation}

By cofactor expansion down the first row, 
\begin{equation} \label{eq:evalform}
\begin{split}
\det(\widetilde{X} - \lambda I) &= (c_1 - \lambda) \det \begin{pmatrix} c_2 - \lambda &     -1 & \cdots  & 0\\
        & (c_3 - \lambda) & -1 & \\
 0 & \cdots & (c_{n-1} - \lambda) & -1\\
      0 & \cdots &       0 & (c_n - \lambda) \end{pmatrix}\\
      & \qquad   + (-1)^{n-1} \cdot (-1) \cdot \det \begin{pmatrix} -1 & 0 & \dotsc & 0\\
0 & -1 & 0 & 0\\
0 & 0 & \ddots & 0\\
0 & \dotsc & 0 & -1
\end{pmatrix}\\
& =  \prod_{k = 1}^n (c_k - \lambda) + (-1)^n (-1)^{n-1}\\
& = \prod_{k = 1}^n (c_k - \lambda) - 1.
\end{split}
\end{equation}
As a matter of technical convenience for the later proof, we rewrite $-\lambda$ as $z$.  Thus, we analyze the roots of an algebraic equation 
\begin{equation}\label{eq:maineqn}
P(z) = 1, \quad \text{where} \quad P(z) \declare P(z; \cvec) =  \prod_{k = 1}^n (c_k + z).
\end{equation}
We define
\begin{subequations} \label{eq:Esets}
\begin{align}
E^- &\declare \bracepair{\xi \in \CC: \Re \xi < 0},\\
E^0 &\declare \bracepair{\xi \in \CC: \Re \xi = 0},\\
E^+ &\declare \bracepair{\xi \in \CC: \Re \xi > 0},
\end{align}
\end{subequations}
and
\begin{equation} \label{eq:Eext}
E \declare \overline{E^+} = \bracepair{\xi \in \CC: \Re \xi \geq 0}.
\end{equation}
Let $\countnroots{\cvec}$ (respectively $\countzroots{\cvec}$, $\countproots{\cvec}$, $\countpzroots{\cvec}$) denote the number of solutions to \eqref{eq:maineqn} in $E^-$ (respectively $E^0$, $E^+$, $E$), counted with multiplicity.

If in the above construction, $X = X_* \declare \alpha I - \beta \permmatspecial$, i.e.,
\begin{equation} \label{eq:idealmatrix}
X_*= \begin{pmatrix}
   \alpha & -\beta &     0  & \cdots  & 0\\
     0 &   \alpha & -\beta & \cdots  & 0\\
\vdots &        & \ddots & \ddots  & \\
     0 & \cdots & \cdots & \alpha & -\beta\\
-\beta &      0 & \cdots &       0 & \alpha
\end{pmatrix} \, \, ,
\end{equation}
then $\widetilde{X} = \cvecgeom I - \permmatspecial$, and the characteristic polynomial of $\widetilde{X}$ is
\begin{equation}
\det(z + \widetilde{X}) = (z + \cvecgeom)^n - 1, \quad \cvecgeom = \frac{\alpha}{\beta}.
\end{equation}
Letting $\cvecideal \declare (\cvecgeom, \cvecgeom, \dotsc, \cvecgeom)$, we have the algebraic equation 
\begin{equation} \label{eq:maineqnspecial}
P^*(z) = 1, \quad \text{where} \quad P^*(z) \declare P(z; \cvecideal) = (\cvecgeom + z)^n.
\end{equation}

The set of solutions to \eqref{eq:maineqnspecial} is
\begin{equation} \label{eq:specialsolform}
\bracepair*{- \cvecgeom + \omega^k: 0 \leq k < n}, \quad \text{where} \quad \omega = \exp \left(\frac{2\pi i}{n}\right),
\end{equation}
and therefore we have (for $\cvecgeom > 0$)
\begin{subequations} \label{eq:specialsolcountform}
\begin{align}
\countproots{\cvecideal} & = \# \bracepair*{k: 0 \leq k < n: \cos \left( \frac{2 \pi k}{n} \right) > \cvecgeom} \, ,\\
\countpzroots{\cvecideal}& = \# \bracepair*{k: 0 \leq k < n: \cos \left(\frac{2 \pi k}{n} \right) \geq \cvecgeom}\, .
\end{align}
\end{subequations}

Thus, we have reduced the main theorem to the following proposition, as we are interested in counting the number of roots of \eqref{eq:maineqn} and \eqref{eq:maineqnspecial} with positive real part.  

\begin{thm} \label{thm:maincount}
Fix $n \in \NN$.  If $\cvec = \cveclong \in (\rplus)^n$, then
\begin{equation} \label{eq:countineq}
\countproots{\cvec} \leq \countproots{\cvecideal}, \quad \countpzroots{\cvec} \leq \countpzroots{\cvecideal}.
\end{equation}
\end{thm}

In addition, we may describe the range of roots of \eqref{eq:maineqn} in the open or closed right-half-plane.  To state the results succintly, note by \eqref{eq:specialsolcountform} that the number of solutions to \eqref{eq:maineqnspecial} is either $0$ or odd, by the evenness of the cosine function, and is nonzero if $\cvecgeom < 1 = \cos(0)$; therefore, if $\cvecgeom < 1$, we may write for some $\kappa_+$, $\overline{\kappa} \in \NN$ that

\begin{subequations} \label{eq:kappaformsIntro}
\begin{align}
\countproots{\cvecideal} &= 2 \kappa_+ + 1\\
\countpzroots{\cvecideal} &= 2 \overline{\kappa} + 1.
\end{align}
\end{subequations} 

\begin{thm} \label{thm:polyCountRangeFull}
Fix $n \in \NN$ and $\cvecgeom \in \RR^+$.  Then the range of $\countproots{\cvec}$ among the set of $\cvec \in \rplus^n$ with geometric mean $\cvecgeom$ is:
\begin{equation}
\begin{cases}
\bracepair{0}, & \text{ if } \cvecgeom \geq 1\\
2 \bracepair{0, 1, \dotsc, \kappa_+} + 1, & \text{ if } \cvecgeom < 1.
\end{cases}
\end{equation} 
Similarly, the range of $\countpzroots{\cvec}$ among these vectors is 
\begin{equation}
\begin{cases}
\bracepair{0}, & \text{ if } \cvecgeom > 1\\
\bracepair{1}, & \text{ if } \cvecgeom = 1\\
2 \bracepair{0, 1, \dotsc, \overline{\kappa}} + 1, & \text{ if } \cvecgeom < 1.
\end{cases}
\end{equation}  
\end{thm}

\begin{thm} \label{thm:matrixCountRangeFull}
Fix $n \in \NN$ and $\alpha, \beta \in \RR^+$.  Then:
\begin{enumerate}[label = (\alph*)]
\item If $\alpha > \beta$, then no $X \in DC(\alpha, \beta)$ has an eigenvalue in the closed left half-plane.
\item If $\alpha = \beta$, then for every $X \in DC(\alpha, \beta)$, $0$ is the only eigenvalue in the closed left half-plane.  
\item If $\alpha < \beta$, then $X \in DC(\alpha, \beta)$ has an odd number of eigenvalues in the open left half-plane, but no more than that of $\alpha I - \beta \permmatspecial$.  Moreover, for every such odd number $k$, some $X \in DC(\alpha, \beta)$ has exactly $k$ eigenvalues in the open left half-plane.  Similarly for the closed left half-plane. 
\end{enumerate} 
\end{thm}

The core of the paper (Sections~\ref{sec:tech} to \ref{sec:movement}) is devoted to proving Theorem~\ref{thm:maincount}.  First, we observe that the roots of $P(z; \cvec) = 1$ in the right-half-plane are \emph{simple},and are bounded away from $\infty$ and $0$ with bounds only depending on $\max_j c_j$, $\min_j c_j$, and $\cvecgeom$; these statements are recorded in Section~\ref{sec:tech}.  Moreover, their number is odd.  This allows us to show that the zeroes in the region of interest vary smoothly as $\cvec$ varies, indeed to use the Implicit Function Theorem (our variation is described in Appendix~\ref{sec:IFT}).  We therefore wish to find a path $\cvec(t)$ starting from any $\cvec_0$ to $\cvecideal$, along which $\countproots{\cvec(t)}$ is increasing.  We are still wondering if a ``direct'' path would work, but we choose to build it step-by-step, steadily bringing the most extreme elements to meet with the next most extreme.  Our rephrasing in terms of the number of \emph{distinct} elements of $\cvec$, and creating a path from any given $\cvec_0$ to one with less extreme $\max_j c_j$ and $\min_j c_j$ (and fewer distinct elements) is related in Section~\ref{sec:ctod}.  In Section~\ref{sec:method}, we study the effects on the roots with positive real part, showing that they remain in the right half-plane.  In the beginning of Section~\ref{sec:movement}, we put the partial paths together to build the desired path from $\cvec_0$ to $\cvecideal$ along which the number of roots of \eqref{eq:maineqn} with positive (or 0) real part is increasing.  Appendix~\ref{sec:poscond} clarifies a positivity condition used in this work.  .

The end of Section~\ref{sec:movement}, and Sections~\ref{sec:nuOne}, \ref{sec:nuAll}, present more details about the precise behavior of the zero-counting functions, and complete the proofs of Theorems~\ref{thm:polyCountRangeFull} and \ref{thm:matrixCountRangeFull}.

The remaining sections tighten the bounds on the range of permissible zeroes in the right-half plane, giving dimension-invariant bounds.  Section~\ref{sec:furtherOne} gives the details, and Appendix~\ref{sec:furthertwo} clarifies a bound used in this work.

\section{Technical Preliminaries} \label{sec:tech}
For $z = x + iy$, \eqref{eq:maineqn} implies
\begin{subequations} 
\begin{gather}
\abs{\prod_{k = 1}^n (z + c_k)} = 1, \label{eq:maineqnabs}\\
\abs{ \prod_{k = 1}^n (z + c_k)}^2 = \prod_{k = 1}^n \left( (x + c_k)^2 + y^2 \right) = 1. \label{eq:maineqnabstwo}
\end{gather}
\end{subequations}
If $z \in E \in \eqref{eq:Eext}$, $z \neq 0$, it follows from \eqref{eq:maineqnabstwo} that 
\begin{equation} \label{eq:toobig}
1  > \prod_{k = 1}^n c_k^2 = \cvecgeom^{2n} ;
\end{equation}
if $z = 0$,
\begin{equation} \label{eq:toobigparttwo}
1 = \prod_{k = 1}^n c_k^2 =  \cvecgeom^{2n}. 
\end{equation}
Therefore, if $\cvecgeom > 1$, then $\countpzroots{\cvec} = 0$, and with $\cvecgeom = 1$, the point $z^* = 0$ is the only solution for \eqref{eq:maineqn} and \eqref{eq:maineqnspecial} in $E$.  In both cases,
\begin{equation}
\countproots{\cvec} \leq \countproots{\cvecideal}, \quad \countpzroots{\cvec} \leq \countpzroots{\cvecideal},
\end{equation} 
i.e., \eqref{eq:countineq} holds.  

In the sequel, we therefore analyze only the case
\begin{equation} \label{eq:gamrestrict}
0 < \cvecgeom < 1.  
\end{equation}
In this case, $\countproots{\cvec} \geq 1$, since all coefficients are real, and
\[
P(0, \cvec) =  \cvecgeom^n < 1 < \prod_{k = 1}^n (1 + c_k) = P(1, \cvec).
\]
We further note that the function is strictly increasing on $[0, \infty)$, so this root is simple, and unique on $[0, \infty)$.  

If $P(w; \cvec) = 1$, then by conjugation,
\[
P(\overline{w}; \cvec) = 1,
\]
and so $\overline{w}$ is also a root of \eqref{eq:maineqn} (of the same multiplicity).  Similarly, if we consider
\[
h(y) = \abs{P(iy, \cvec)}^2 = \prod_{k = 1}^n (y^2 + c_k^2),
\]
we have that
\[
h(0) = \cvecgeom^{2n} < 1 < \prod_{k = 1}^n (1 + c_k^2) = h(1),
\]
and that $h$ is even, and increasing on $[0, \infty)$, so there exists a unique solution on the positive imaginary axis to \eqref{eq:maineqnabstwo} (i.e., $x = 0$); call it $z = i Y(\cvec)$.  Therefore, we have the following.
\begin{lem} \label{lem:oddcount}
Fix $\cvec \in (\rplus)^n$ with geometric mean $\cvecgeom < 1$.  Then $\countproots{\cvec}$ and $\countpzroots{\cvec}$ are both odd and positive.  $P(z; \cvec)$ has exactly one root in $(0, 1)$, and the others are not real.  

Also, $\abs{P(z; \cvec)}^2 = 1$, or $\abs{P(z; \cvec)} = 1$, has a unique solution on the positive imaginary axis.
\end{lem}
For $\cvecgeom, \cvecmin, \cvecmax$, satisfying 
\begin{equation} \label{eq:gamminmaxineqs}
0 < \cvecmin \leq \cvecgeom \leq \cvecmax,
\end{equation}
we define
\begin{equation} \label{eq:cspacedef}
\cspace{\cvecgeom}{\cvecmin}{\cvecmax} = \bracepair*{\cvec = \cveclong \in [\cvecmin, \cvecmax]^n : \prod_{k = 1}^n c_k = \cvecgeom^n}.
\end{equation}  
\begin{smallrem} \label{rem:wlogrem}
In the analysis of polynomials $P(z, \cvec)$ and related algebraic equations, without loss of generality, we may suppose that the $c_k$ are in order, i.e.
\begin{equation} \label{eq:cbounds}
0 < \cvecmin \leq c_1 \leq c_2 \leq \dotsc \leq c_m \leq \cvecmax < \infty.
\end{equation}  
It will be useful in the technical analysis which follows.  But it helps to understand that in Theorem~\ref{thm:eigencounts}, we can talk about \emph{any} $\permmat$, not just $\permmatspecial$, which corresponds to an $n$-cycle permutation $\kappa$.

Indeed, for $A = \diag(a_1, \dotsc, a_n)$,
\begin{equation}
\det \left[(A - \lambda I) - \beta \permmat \right] = \prod_{k = 1}^n (a_k - \lambda) + (-1)^{\sgn \kappa} (-\beta)^n,
\end{equation} 
and
\begin{equation}
\sgn \kappa = n - 1;
\end{equation}
see, e.g., \cite[Section 3.5, p. 110]{DuFo}, either Proposition 25 or the line +13.

\end{smallrem}

For $\cvec \in \cspace{\cvecgeom}{\cvecmin}{\cvecmax}$, we may uniformly establish a root-free zone for $P(z; \cvec)$ in a small disk centered at the origin.

\begin{lem} \label{lem:zerofreedisk}
Fix $0 < \cvecgeom < 1$ and two positive real numbers $\cvecmin$ and $\cvecmax$, satisfying \eqref{eq:gamminmaxineqs}.  Then for all $\cvec \in \cspace{\cvecgeom}{\cvecmin}{\cvecmax}$, there are no roots to \eqref{eq:maineqn} in the closed disk $\bracepair{\xi \in \CC: \abs{\xi} \leq d}$, where
\begin{equation} \label{eq:ddef}
d \declare d(\cvecgeom, \cvecmax) =  (1 - \cvecgeom^n) \left( 1 + \cvecmax \right)^{-n}.
\end{equation}
\end{lem}
\begin{proof}
We denote  by $\sympoly{j}{\cvec} = \sympoly{j}{(c_1, \dotsc, c_n)}$ the $j$th elementary symmetric polynomial evaluated at $\cveclong$.  If $z$ is a root of \eqref{eq:maineqn}, then $z \neq 0$ because $\cvecgeom < 1$, as $\prod_{k = 1}^n c_k = \cvecgeom^n < 1$.  Then by \eqref{eq:maineqnabs}, and Lemma~\ref{lem:basicUB}
\begin{equation}
\begin{split}
1 = \abs{\prod_{k = 1}^n (z + c_k)} & \leq  \prod_{k = 1}^n (\abs{z} + c_k)\\
& = \sum_{k = 0}^n \sympoly{n - k}{\cvec} \abs{z}^k\\
& = \sympoly{n}{\cvec} + \sum_{k = 1}^n \sympoly{n-k}{\cvec} \abs{z}^k\\
& = \cvecgeom^n + \abs{z} \sum_{k = 1}^n \sympoly{n-k}{\cvec} \abs{z}^{k - 1}\\
& < \cvecgeom^n + \abs{z} \left[(\cvecmax)^n + \sum_{k = 1}^n \sympoly{n-k}{(\cvecmax, \cvecmax, \dotsc , \cvecmax)} 1 \right]\\
& = \cvecgeom^n + \abs{z} (1 + \cvecmax)^n,
\end{split}
\end{equation}
so
\begin{equation} \label{eq:rewrite}
\frac{1 - \cvecgeom^n}{(1 + \cvecmax)^n} < \abs{z}. 
\end{equation} 
\end{proof}

For $\cvecmin > 0$, define the closed half-plane
\begin{equation} \label{eq:eextdef}
\extplane = \extplane(\cvecmin) \declare \bracepair*{\xi \in \CC: \Re \xi \geq - \, \frac{\cvecmin}{3}}
\end{equation}

For $\cvec \in \cspace{\cvecgeom}{\cvecmin}{\cvecmax}$, we can also bound from above the size of the roots of \eqref{eq:maineqn} in $\extplane$.

\begin{lem} \label{lem:basicUB}
Fix $n \in \NN$ and two positive real numbers $\cvecmin$ and $\cvecmax$, satisfying \eqref{eq:gamminmaxineqs}.  Then for all $\cvec \in \cspace{\cvecgeom}{\cvecmin}{\cvecmax}$, all roots of \eqref{eq:maineqn} in $\extplane(\cvecmin)$ are in the disk $\bracepair{\xi \in \CC: \abs{\xi} < 1}$.  
\end{lem} 
\begin{proof}
If $z = x + iy$ is a root of \eqref{eq:maineqn} with $x \geq 0$, $y$ real, then by \eqref{eq:maineqnabstwo},
\begin{equation} \label{eq:absabovebasic}
\begin{split}
1 = \abs{\prod_{k = 1}^n (z + c_k)}^2 & = \prod_{k = 1}^n \left( (x + c_k)^2 + y^2 \right)\\
& = \prod_{k = 1}^n \left((x^2 + y^2) + c_k (2 x + c_k)\right)\\
& > \prod_{k = 1}^n (\abs{z}^2 + c_k \left[ - \, \frac{ 2\cvecmin}{3}  + \cvecmin \right]) > \abs{z}^{2n},
\end{split}
\end{equation}
so $\abs{z} < 1$.  (Indeed, $x > - \, \frac{\cvecmin}{2}$ is all that is required here).  
\end{proof}
(In Section~\ref{sec:furtherOne} and Appendix~\ref{sec:furthertwo} we give better estimates, but Lemma~\ref{lem:basicUB} is good enough for the proof of our main theorem.)  

We define $\ann{r}{R} = \bracepair{z \in \CC: r < \abs{z} < R}$, the annulus centered at the origin with  radii $r$ and $R$.  We summarize Lemmas~\ref{lem:basicUB} and \ref{lem:zerofreedisk} as follows.
\begin{cor} \label{cor:zerocontainmentannulus}
Fix $0 < \cvecgeom <1$, and $\cvecmin, \cvecmax$ positive reals with $\cvecmin \leq \cvecgeom \leq \cvecmax$.  Then for any $\cvec$ in $\cspace{\cvecgeom}{\cvecmin}{\cvecmax}$, all zeroes of \eqref{eq:maineqn} in $\extplane(\cvecmin)$ are also in $\ann{d}{1}$, $d \in \eqref{eq:ddef}$.  
\end{cor}

In the sequel, it is sometimes more convenient to use a bounding box, rather than a bounding semiannulus, for the permissible range of the zeros with positive real part.  
\begin{cor} \label{cor:zerofreebox}
Fix $0 < \cvecgeom < 1$ and two positive real numbers $\cvecmin$ and $\cvecmax$, satisfying \eqref{eq:gamminmaxineqs}.  Then for all $\cvec \in \cspace{\cvecgeom}{\cvecmin}{\cvecmax}$, if $z \in \extplane(\cvecmin)$ satisfied $P(z, \cvec) = 1$, then $w$ is inside the box
\begin{equation} \label{eq:ourboxdef}
\ourbox =  \bracepair*{\xi \in \CC: - \frac{\cvecmin}{3} < \Re \xi < 1, \abs{\Im \xi} < 1} \setminus \bracepair*{\xi \in \CC: \abs{\Re \xi} \leq \delta, \abs{\Im \xi} \leq \delta},
\end{equation}
where  
\begin{equation} \label{eq:deltadef}
\delta = \delta(\cvecgeom, \cvecmax) \declare \frac{2}{3} d = \frac{2}{3}   (1 - \cvecgeom^n) \left( 1 + \cvecmax \right)^{-n}  \leq \frac{2}{3}.
\end{equation}
\end{cor}
\begin{proof}
If $z = x + iy$ with $x > - \frac{\cvecmin}{3}$, $y$ real, then by Lemma~\ref{lem:basicUB}, $\abs{z} < 1$, so $x < 1$ and $\abs{y} < 1$.  If $\abs{x} \leq \delta$ and $\abs{y} \leq \delta$, $\abs{z}^2 = x^2 + y^2 \leq 2 \delta^2 = 2 \cdot \left( \frac{2}{3} d\right)^2 = \frac{8}{9} d^2 < d^2$, so $\abs{z} < d$ and we may apply 
Lemma~\ref{lem:zerofreedisk}.
\end{proof}
\begin{figure}
\begin{tikzpicture}[scale = 3.0]
\draw [<->](-1.1, 0) -- (1.1, 0) node[right] {$x = \Re z$};
\draw [<->] (0, -1.1) -- (0, 1.2) node[above] {$y = \Im z $}; 
\draw [very thick] (-1/16, -0.9961) arc (-93.5833:93.5833:1) -- (-1/16, 0.1768) arc(109.4712:-109.4712:3/16) -- (-1/16, -0.9961);
\filldraw[fill = blue] (-1/16, -0.9961) arc (-93.5833:93.5833:1) -- (-1/16, 0.1768) arc(109.4712:-109.4712:3/16) -- (-1/16, -0.9961);
\node at (-0.5, 0.5) {$\ann{d}{1}$};
\draw [->] (-0.2, 0.5) -- (0.2, 0.5);
\draw [very thick] (-1/16, -1) -- (1, -1) -- (1, 1) -- (-1/16, 1) -- (-1/16, 1/8) -- (1/8, 1/8) -- (1/8, -1/8) -- (-1/16, -1/8) -- (-1/16, -1);
\node at (1.2, 1.2) {$\ourbox$};
\draw [->] (1.1, 1.1) -- (0.9, 0.9);
\end{tikzpicture}
\label{fig:zerofrees}
\caption{A drawing of the restricted regions for the roots of $P(z; \cvec) = 1$ in $\extplane(\cvecmin)$, for $\cvec \in \cspace{\cvecgeom}{\cvecmin}{\cvecmax}$.}
\end{figure}
The containing regions for the roots of $P(z; \cvec) = 1$ are shown in Figure 1.

We will frequently use --- even without a reference --- the following.
\begin{rem} \label{rem:inverserealsgn}
If $z \in \CC$, $z \neq 0$, then $ \Re \frac{1}{z}  > 0$ (respectively, $\Re \frac{1}{z} < 0$) if and only if $\Re z > 0$ (respectively, $\Re z < 0$).
\end{rem}
\begin{proof}
It immediately follows from the identity
\begin{equation} \label{eq:invertersign}
\Re \frac{1}{z} = \Re \frac{\overline{z}}{\abs{z}^2} = \frac{1}{\abs{z}^2} \Re \overline{z} = \frac{x}{\abs{z}^2}, \quad z \neq 0, \quad z = x + iy, \quad x, y \text{ real}.
\end{equation}
\end{proof}

\begin{lem} \label{lem:simplezeros}
Fix $0 < \cvecgeom$, and let $\cvecmin, \cvecmax \in \rplus$ satisfy $\cvecmin \leq \cvecgeom \leq \cvecmax$.  Then for all $\cvec \in \cspace{\cvecgeom}{\cvecmin}{\cvecmax}$, if $w$ is a root of $\eqref{eq:maineqn}$ and $w \in \extplane(\cvecmin)$, then $w$ is a simple root of \eqref{eq:maineqn}.  
\end{lem}
\begin{proof}
Let $w$ be a root of \eqref{eq:maineqn} in $\extplane(\cvecmin)$.  Then for all $k$, $1 \leq k \leq n$, 
\begin{equation} \label{eq:numeratorineq}
\Re (w + c_k) \geq \Re (w + \cvecmin) \geq \frac{2\cvecmin}{3},
\end{equation}
and by \eqref{eq:maineqn}, \eqref{eq:invertersign} and Lemma~\ref{lem:basicUB}, 
\[
\Re \left( \left. \frac{dP}{dz} \right\vert_{z = w} \right)  = \sum_{k = 1}^n \Re \left(\frac{1}{w + c_k} \right) = \sum_{k = 1}^n \frac{\Re (w + c_k)}{\abs{w + c_k}^2} \geq n \frac{2\cvecmin/3}{(1 + \cvecmax)^2} > 0
\]
so $w$ is a simple root of \eqref{eq:maineqn}.
\end{proof}

\section{Restructing of the sequence \texorpdfstring{$\cvec$}{c}} \label{sec:ctod}
Under \eqref{eq:cbounds}, we will not change the orders of the elements in the vectors, and if $c_k = c_{k + 1}$ are identical, we will never do any change to make them nonequal.  Therefore, it now behooves us only pay attention to the \emph{distinct} entries in $\cveclong$.  We choose to write the distinct entries in $\cveclong$ as $(d_0, \dotsc, d_{q})$, 
\begin{equation} \label{eq:dlist}
0 < \cvecmin \leq d_0 < d_1 < \dotsb < d_q \leq D^*,
\end{equation}
so that the number of distinct entries is $1 + q$; i.e., the number of strict inequalities in \eqref{eq:cbounds}, or the number of gaps in \eqref{eq:dlist} is $q$.  Hereafter, we call $q$ the \emph{\diver} of the multiset.  We therefore reformulate $\cvec \in \CC^n$ as a multiset with $q + 1$ distinct entries and total weight $n$.  For a given $\cvec$, let
\begin{equation} \label{eq:klist}
\begin{split}
K_j & \declare \bracepair*{k \in \setn: c_k = d_j}, \, \text{ where } \setn = \bracepair{1, 2, \dotsc, n} \\
&\left( \text {so } \bigcup_{j = 0}^q K_j = \setn , \quad K_j \cap K_{\ell} = \emptyset \text{ if } j \neq \ell \right), 
\end{split}
\end{equation}
and let 
\begin{equation} \label{eq:mlist}
m_j = \card{K_j}, \quad 0 \leq j \leq q, \left( \text{so } \sum_{j = 0}^q m_j = n \right).
\end{equation}
In short, we present $\cvec$ as $\dshort{\mathbf{d}}{\mathbf{m}}{q} = \dlong{d_j}{m_j}{q}$.  In this language, we have
\begin{equation} \label{eq:ourpolydform}
P(z; \mathbf{c}) = \prod_{k = 1}^n (1 + c_k) = \prod_{j = 0}^q (1 + d_j)^{m_j},
\end{equation}
\begin{equation}
\cvecgeom^n = \prod_{k = 1}^n c_j = \prod_{j = 0}^q d_j^{m_j},
\end{equation}
and $\cvecideal$ is presented by $\mathbf{d}^* = \dshort{\cvecgeom}{n}{0}$.  The family $\cspace{\cvecgeom}{\cvecmin}{\cvecmax}$ is presented by the family of multisets
\begin{equation} \label{eq:dspacedef}
\begin{split}
\dspace{\cvecgeom}{\cvecmin}{\cvecmax} &= \bigcup_{q = 0}^{n - 1} \left\lbrace (d_j, m_j)_{j = 0}^q \in ([\cvecmin, \cvecmax] \times (\NN \cup \bracepair{0}))^{q + 1} : \right.\\
& \qquad \left. \prod_{j = 0}^q d_j^{m_j} = \cvecgeom^n, \quad \sum_{j = 0}^q m_j = n \right\rbrace
\end{split} 
\end{equation}

We now construct the basic elements of our path connecting $\dshort{\mathbf{d}}{\mathbf{m}}{q}$ to $\dgroupideal$, or $\cvec$ to $\cvecideal$.  Our goal is to reduce the \diver\ $q$, i.e., the number of gaps, and maintain the geometric mean.  

If $q \geq 2$, we put
\begin{subequations} \label{eq:taufulldef}
\begin{align}
\tau_* = \frac{1}{m_q} \log \left( \frac{d_1}{d_0} \right) \label{eq:taulowdef}\\
\tau^* = \frac{1}{m_0} \log \left( \frac{d_q}{d_{q - 1}} \right)\\
\tau = \min \bracepair{\tau_*, \tau^*} \label{eq:qgentaudef}.
\end{align}
\end{subequations}
We have three cases:
\begin{enumerate}[label = (\Roman*)]
\item \label{enum:lowersmaller} $\tau = \tau_* < \tau^*$,
\item \label{enum:uppersmaller} $\tau_* > \tau^* = \tau$,
\item \label{enum:luequal} $\tau_* = \tau^* = \tau$. 
\end{enumerate}

We now construct a path for the sequence \eqref{eq:cbounds}, or for the multiset \eqref{eq:dlist}, parameterized by $t$ in $[0, \tau)$ and $[0, \tau]$.  We define on $[0, \tau)$
\begin{equation}\label{eq:dkdefhgen}
d_j(t) = \begin{cases} d_0 \exp(m_q t), & j = 0,\\
d_j, & 0 < j < q\\
d_q \exp(-m_0 t), & j = q \end{cases} .
\end{equation} 
We note that the multiplicities are unchanged on $[0, \tau)$: for $0 < j < q$, the $d_j$ do not move, and for $0 < t < \tau \leq \tau_*$, by \eqref{eq:taulowdef},
\[
d_0 < d_0(t) = d_0 \exp (m_q t) < d_0 \exp (m_q \tau_*) = d_1;
\]
similarly, $d_{q}(t) > d_{q - 1}(t)$ for $t$ in $[0, \tau)$.  The geometric mean is preserved:
\begin{equation} \label{eq:geompreserve}
\begin{split}
\prod_{j = 0}^q d_j(t)^{m_j} &= (d_0 \exp (m_q t))^{m_0} \cdot \left( \prod_{j = 1}^{q - 1} d_j^{m_j} \right) \cdot (d_q \exp ( - m_0 t))^{m_q}\\
& = \left( \prod_{j = 0}^q d_j^{m_j}\right)  \cdot \exp(m_q m_0 t - m_0 t m_q) = \prod_{j = 0}^q d_j^{m_j} = \cvecgeom^n.
\end{split}
\end{equation}
Since $d_0(t)$ is increasing and $d_q(t)$ is decreasing, we have that if $\dshort{\mathbf{d}(0)}{\mathbf{m}}{q}$ is in $\dspace{\cvecgeom}{\cvecmin}{\cvecmax}$, then so is $\dshort{\mathbf{d}(t)}{\mathbf{m}}{q}$ for all $t$ in $(0, \tau)$.  The cases differ in the appropriate extension when $t = \tau$.  
\begin{enumerate}[label = (\Roman*)]
\item In this case, $\tau = \tau_*$, so $\lim_{t \to \tau^-} d_0(t) = d_1$, but $\tau \neq \tau^*$, so $\lim_{t \to \tau^-} d_q(t) = d_q \exp (- m_0 \tau_*) > d_{q-1}$.  Therefore, the end multiset $\bracepair{(d_j^{\prime}, m_j^{\prime})}_{j = 0}^{q^{\prime}}$ at $t = \tau$ is defined with
\begin{equation} \label{eq:caseOnefinaldisposition}
\begin{gathered}
q^{\prime} = q - 1,\\
 m_0^{\prime} = m_0 + m_1, \quad m_j^{\prime} = m_{j = 1}, 1 \leq j \leq q^{\prime},\\
d_0^{\prime} = d_1 = d_0 \exp(m_q \tau), \, \, d_j^{\prime} = d_{j + 1} \text{ for }1 \leq j < q^{\prime}, d^{\prime}_{q^{\prime}} = d_q \exp (- m_0 \tau).\\
\end{gathered}
\end{equation} 
In short, the $0$th and $1$st points of the multiset \eqref{eq:dlist} have coalesced.  Again, the geometric mean is $\cvecgeom$, by continuity, and the end multiset belongs to $\dspace{\cvecgeom}{\cvecmin}{\cvecmax}$.  
\item In this case, $\tau = \tau^*$ so $\lim_{t \to \tau^-} d_q(t) = d_0 \exp(m_q \tau^*)= d_{q-1}$, but $\tau \neq \tau_*$, so $\lim_{t \to \tau^-} d_0(t) < d_{1}$.  Therefore, the end multiset $\bracepair{(d_j^{\prime}, m_j^{\prime})}_{j = 0}^{q^{\prime}}$ at $t = \tau$ is defined with
\begin{equation} \label{eq:caseTwofinaldisposition}
\begin{gathered}
q^{\prime} = q - 1,\\
 m_j^{\prime} = m_j, 0 \leq j < q^{\prime}, \quad m^{\prime}_{q^{\prime}} = m_{q - 1} + m_q,\\ 
d_0^{\prime} = d_0 \exp(m_q \tau), \, \, d_j^{\prime} = d_{j + 1} \text{ for }1 \leq j < q^{\prime}, d^{\prime}_{q^{\prime}} = d_{q-1} = d_q \exp (- m_0 \tau).\\
\end{gathered}
\end{equation} 
In short, the $(q-1)$st and $q$th points of the multiset \eqref{eq:dlist} have coalesced.  Again, the geometric mean is $\cvecgeom$, by continuity, and the end multiset belongs to $\dspace{\cvecgeom}{\cvecmin}{\cvecmax}$.
\item In this case, we have both the lowest $2$ and upper $2$ points of the multiset \eqref{eq:dlist} coalescing.  It behooves us to separate out the case $q > 2$ (so $d_1 \neq d_{q-1}$) and $q = 2$ (where $d_1 = d_{q - 1}$).  
	\begin{enumerate}[label = (\alph*)]
	\item \label{enum:caseA}
	If $q > 2$, then $q^{\prime} = q - 2$, and the end multiset $\bracepair{(d_j^{\prime}, m_j^{\prime})}_{j = 0}^{q^{\prime}}$ at $t = \tau$ is defined with
	\begin{equation} \label{eq:caseThreeAfinaldisposition}
	\begin{gathered}
	 m_0^{\prime} = m_0 + m_1, \quad m_j^{\prime} = m_{j + 1}, \text{ for } 2 \leq j < q^{\prime}, \quad m^{\prime}_{q^{\prime}} = m_{q - 1} + m_q,\\ 
	d_0^{\prime} = d_1 = d_0 \exp(m_q \tau), \, \, d_j^{\prime} = d_{j + 1} \text{ for } 1 \leq j < q^{\prime}, \quad d^{\prime}_{q^{\prime}} = d_{q-1} = d_q \exp (- m_0 \tau).\\
	\end{gathered}
	\end{equation} 
	\item \label{enum:caseB} If $q = 2$, then $q^{\prime} = 0$, and $m_0^{\prime} = m_0 + m_1 + m_2$, $d_{0}^{\prime} = d_1$.  Since the geometric mean is preserved we must have $\dgroupideal$, a multi-singleton, our goal.
	\end{enumerate}	  
\end{enumerate}

Finally, we handle the $q = 1$ case.
\begin{enumerate}[label = (\Roman*),resume]
\item \label{enum:qone} If $q = 1$, we find $\tau > 0$ such that
\begin{equation} \label{eq:qonetaudef}
\begin{gathered} d_0 \exp(m_1 \tau) = d_1 \exp(- m_0 \tau), \quad \text{i.e.,}\\
\tau = \frac{1}{n} \log \frac{d_1}{d_0} \end{gathered}
\end{equation}
and set 
\begin{equation} \label{eq:dpathqone}
d_0(t) = d_0 \exp(m_1 t), \quad d_1(t) = d_1 \exp(-m_0 t), \quad 0 \leq t < \tau.
\end{equation}
For $t = \tau$, we change to the multi-singleton $\dshort{d_0^{\prime}}{n}{0}$, $d_0^{\prime} = d_0 \exp(m_1 \tau) = d_1 \exp(- m_0 \tau)$.  Again, since this process does not change the geometric mean, we end up at $\dgroupideal = \dshort{\cvecgeom}{n}{0}$.  
\end{enumerate}
If we wish to speak in terms of $\bracepair{\cvec(t)}$ we always follow \eqref{eq:dlist}-- \eqref{eq:mlist} so
\begin{equation} \label{eq:cdinvert}
c_k(t) = d_j(t), \quad k \in K_j, \quad 0 \leq j \leq q.
\end{equation}

On each step, the coordinates of $\cvec(t)$ have the structure $B \exp (\beta t)$ with $0 < B \leq \cvecmax$ and $\abs{\beta} \leq n$, so the following condition holds:
\begin{equation} \label{eq:cderbounds}
\abs{\cvec^{\prime}(t)} \leq n \cvecmax, \quad \abs{ \cvec^{\prime \prime}(t)} \leq n^2 \cvecmax.  
\end{equation}

We summarize our desired reduction of steps as follows.
\begin{prop} \label{prop:varypropC}
Fix positive real numbers $0 < \cvecmin \leq \cvecgeom \leq \cvecmax$, with $\cvecgeom < 1$.  Fix $\cvec_0 \in \cspace{\cvecgeom}{\cvecmin}{\cvecmax}$, $\cvec_0 \neq \cvecideal$.  Then with $\tau$ defined as in \eqref{eq:qonetaudef} if $q = 1$ and \eqref{eq:qgentaudef} if $q \geq 2$, we have defined a $C^{\infty}$ function $\cvec(t): [0, \tau] \to \cspace{\cvecgeom}{\cvecmin}{\cvecmax}$ such that 
\begin{enumerate}[label = (\alph*)]
\item \label{enum:init} $\cvec(0) = \cvec_0$;
\item \label{enum:hdecr} letting $q^{\prime}$ denote the number of gaps in the $d$-notation for $\cvec(\tau)$,
\begin{equation}
q^{\prime} \leq q - 1.
\end{equation}
Moreover, in Cases~\ref{enum:luequal}.\ref{enum:caseB} and \ref{enum:qone}, $\cvec(t) = \cvecideal$.
\end{enumerate}
\end{prop}

\subsubsection{Extension of Path} \label{subsect:extension}
For the technical arguments later in the paper, we will need to extend the paths $\mathbf{d}(t)$, $\mathbf{c}(t)$ beyond $[0, \tau]$; indeed, for the Implicit Function Theorem, we wish to use complex values for $t$.  Of course, the formulas in \eqref{eq:taufulldef} -- \eqref{eq:dkdefhgen}, \eqref{eq:dpathqone} are valid for all $t \in \CC$, but for any $\rho \in \left( 0, \frac{\log 3}{2n} \right)$, we may simply extend it to the $\CC$-neighborhood
\begin{equation} \label{eq:Jrhodef}
J_{\rho} = \bracepair{\xi \in \CC: - \rho \leq \Re \xi \leq \tau + \rho, \quad \abs{\Im \xi} \leq \rho}
\end{equation} 
Of course, \eqref{eq:geompreserve} still holds, so the geometric mean is preserved, and by the bounds on $\rho$, for any real $r \in [-\rho, \tau + \rho] = J_{\rho} \cap \RR$,
\begin{equation} \label{eq:dzerolowerbound}
d_0 (r) = d_0 \exp(m_q r) \geq d_0 \exp (- m_q \rho) > \cvecmin \exp \left( - n \frac{\log 3}{2n} \right) = \frac{\cvecmin}{\sqrt{3}} 
\end{equation}
and 
\begin{equation} \label{eq:dzeroupperbound}
d_0 (r) = d_0 \exp(m_q r) \leq d_0 \exp (m_q \tau) \exp (m_q \rho)  < d_1 \exp \left( n \frac{\log 3}{2n} \right) \leq \sqrt{3} \cvecmax. 
\end{equation}
Similar bounds hold for $d_q$.  Therefore, for any $t \in \eqref{eq:Jrhodef}$,  we have
\[
\begin{split}
d_0(t) &= d_0 \exp ( m_q t) = d_0 \exp (m_q \Re t + i m_q \Im t)\\
& = d_0 \exp( m_q \Re t) \left[ \cos(m_q \Im t) + i \sin(\Im t) \right].
\end{split}
\]
Of course,
\[
\abs{d_0(t)} = d_0 \exp(m_q \Re t),
\]
so by \eqref{eq:dzerolowerbound} and \eqref{eq:dzeroupperbound},
\[
\frac{\cvecmin}{\sqrt{3}} \leq \abs{d_0(t)} \leq \sqrt{3} \cvecmin,
\]
but we also wish to bound the real and imaginary parts separately.  By $\rho < \frac{\log 3}{2n} < \frac{\pi}{6n}$, $ \abs{m_q \Im t} \leq n \rho <  n \frac{\pi}{6n} = \frac{\pi}{6}$, and so with \eqref{eq:dzerolowerbound}, we have
\[
\Re d_0(t) = d_0 \exp (m_q \Re t) \cos (m_q \Im t) > \frac{\cvecmin}{\sqrt{3}} \cdot \cos \left( \frac{\pi}{6} \right) = \frac{\cvecmin}{2} ,
\]
and by \eqref{eq:dzeroupperbound} we have
\[
\Re d_0(t) =	d_0 \exp (m_q \Re t) \cos (m_q \Im t) \leq \sqrt{3} \cvecmin 
\]
and 
\[
\abs{\Im d_0(t)} = \abs{d_0 \exp(m_q \Re t) \sin (m_q \Im t} \leq  \sqrt{3} \cvecmin \cdot \sin \left( \frac{\pi}{6} \right) = \frac{\sqrt{3}}{2} \cvecmin.
\]
Similar inequalities hold for $d_q(t)$, and if $q > 1$, then $d_1(t), \dotsc, d_{q - 1}(t)$ are still positive constants in $[\cvecmin, \cvecmax]$.  

We create $\cvec(t)$ as in \eqref{eq:cdinvert}, but using the initial $\mathbf{m}$ and $K_j$'s to define the multiplicities, i.e.,
\begin{equation} \label{eq:cvecext}
\cvec(t) = ( \overbrace{d_0(t), \dotsc, d_0(t)}^{m_0 \text{ terms}},  \overbrace{d_1(t), \dotsc, d_1(t)}^{m_1 \text{ terms}}, \dotsc, \overbrace{d_q(t), \dotsc, d_q(t)}^{m_q \text{ terms}} ),
\end{equation}
so that
\begin{equation} \label{eq:polyext}
P(z; \cvec(t)) = \prod_{k = 1}^n (z + c_k(t)) = \prod_{j = 0}^q (z + d_j(t))^{m_j}.
\end{equation}
Of course, for $t > \tau$, $d_0(t) > d_1(t)$ or $d_{q- 1}(t) < d_q(t)$, but \eqref{eq:cvecext} still gives the same polynomials as the previous constructions for $t \in [0, \tau]$, and shows that the \diver\ is at most $q$.  We therefore have the following.
\begin{lem} \label{lem:extensions}
Let $\cvec_0 \in \cspace{\cvecgeom}{\cvecmin}{\cvecmax}$, fix $\rho \in \left( 0, \frac{\log 3}{2n} \right)$, and with $J_{\rho}$ as in \eqref{eq:Jrhodef}, define  $\cvec(t)$ for $t \in J_{\rho}$ as in \eqref{eq:cvecext}, for $\dvec(t)$ as in \eqref{eq:dkdefhgen} for $q > 1$ and \eqref{eq:dpathqone} for $q = 1$. Then $\cvec(t)$ is a holomorphic function on $J_{\rho}$, and the image of  $J_{\rho}$ is inside $\left[R(\cvecmin, \cvecmax)\right]^n$, where
\begin{equation} \label{eq:cextboundingbox}
R(\cvecmin, \cvecmax) = \bracepair*{\xi \in \CC: \frac{\cvecmin}{2} < \Re \xi \leq \sqrt{3} \cvecmax, \, \abs{\Im \xi} \leq \frac{\sqrt{3}}{2} \cvecmax},
\end{equation}
and
\begin{equation} \label{eq:cextabsbounds}
\frac{\cvecmin}{\sqrt{3}} \leq \abs{c_k(t)} \leq \sqrt{3} \cvecmax, 1 \leq k \leq n.
\end{equation} 

In particular, $\cvec(t) \in \cspace{\cvecgeom}{\frac{\cvecmin}{2}}{2\cvecmax}$ for $t \in [- \rho, \tau + \rho]$, and $\cvec(t) \in \cspace{\cvecgeom}{\cvecmin}{\cvecmax}$ for $t \in [0, \tau]$.  
\end{lem}
In Case~\ref{enum:luequal}, part \ref{enum:caseA}, or Case~\ref{enum:qone}, we will not consider $\cvec(t)$ for $t > \tau$.  

\section{Reduction of the \diver\ \texorpdfstring{$q$}{q}} \label{sec:method}
In the $d$-notation \eqref{eq:dlist} -- \eqref{eq:mlist}, our polynomial becomes
\begin{equation} \label{eq:ourpolydformreminder}
P(z, \cvec) = \prod_{j = 0}^q (z + d_j)^{m_j}.
\end{equation}

In Section~\ref{sec:ctod}, we have chosen the path $\cvec(t)$ or $\mathbf{d}(t)$, $0 \leq t \leq \tau$, which reduces the \diver\ $q$ of the initial multiset
\begin{equation}
\mathbf{r} = (r_k)_{k = 1}^n = \cvec(0)
\end{equation}
to $q^{\prime} = q - 1$ or $q - 2$ when we move to $\cvec(\tau)$, i.e., $\mathbf{d}(\tau)$.  The polynomial \eqref{eq:ourpolydformreminder} changes accordingly, and we want to understand how its roots are changing, in particular, when $t$ is close to $0$ or $\tau$.  In what follows, as in Subsection~\ref{subsect:extension}, $\cvec(t)$ is defined by \eqref{eq:dkdefhgen} or \eqref{eq:dpathqone}, i.e. by \eqref{eq:cvecext}, for $- \rho \leq \tau \leq \tau + \rho$, for small enough $\rho$.

\begin{prop} \label{prop:omnibus}
Fix $0 < \cvecmin \leq \cvecgeom \leq \cvecmax$, with $\cvecgeom < 1$.  Let $\mathbf{r} = \cvec(0) \in \cspace{\cvecgeom}{\cvecmin}{\cvecmax}$, and $w \in E$ be a root of the equation 
\begin{equation} \label{eq:polystarttime}
P(z; \mathbf{r}) = 1.
\end{equation}
Then for sufficiently small $\eta > 0$, there exists an unique analytic function $w(t)$, $t \in J_{\eta} \in \eqref{eq:Jrhodef}$, such that
\begin{gather}
w(0) = w,\\
P(w(t), \cvec(t)) = 1, \quad t \in J_{\eta}. \label{eq:solutionGen}
\end{gather}
If $t \in [-\eta, \tau]$, then $w(t) \in E \cap \ann{\frac{d}{2}}{1}$.  If, in addition, $\Re w(t) \in [-\epsilon, \epsilon]$,
\begin{equation} \label{eq:epsilondef}
\epsilon \declare \min \bracepair*{ \frac{\cvecmin}{12},  \frac{\delta(\cvecgeom, \sqrt{3} \cvecmax)^2}{4(1 +  \sqrt{3}\cvecmax)}},
\end{equation}
then
\begin{equation} \label{eq:introwtdersign}
\Re \wtder(t) > 0.  
\end{equation}
\end{prop}

\begin{proof} To use Appendix~\ref{sec:IFT}, Claim~\ref{clm:IFT}, we put 
\begin{subequations} \label{eq:IFTsetup}
\begin{gather}
F(z, t) = P(z; \cvec(t)) - 1, \label{eq:fset}\\
\rho = \frac{1}{2} \min \bracepair*{\epsilon, \frac{\log 3}{2n}}, \label{eq:rhoset}\\
V = \bracepair{z \in \CC: \Re z \geq - \epsilon, \abs{z} \leq 1} , \label{eq:Vset}\\
J = [0, \tau] \label{eq:Jset}
\end{gather}
\end{subequations}
so that the lozenge-shaped neighborhood $J(\rho)$ defined as in $\eqref{eq:rhonbhddef}$ is a subset of the rectangle $J_{\rho} \in \eqref{eq:Jrhodef}$.

We first note the following estimate: If $\abs{z} \leq 2$, and $\abs{c_k(t)} \leq 2 \cvecmax$ for all $k \in \NN$, $k \leq n$, then
\begin{equation} \label{eq:mZerodef}
\abs{P(z; \cvec(t)} \leq 2^n (1 + \cvecmax)^n \equiv M_0
\end{equation}

The estimate is on an appropriate domain: for $z \in V$, $\abs{z} \leq 1$, so for $z \in V_{\rho}$ with $\rho < 1$, $\abs{z} < 2$.  For $t \in J_{\rho}$, Lemma~\ref{lem:extensions}, \eqref{eq:cextabsbounds}, ensures $\abs{c_k(t)} \leq 2 \cvecmax$ for all $k$.

We divide the next part of the proof into smaller claims.

\begin{bigclm} \label{clm:IFTLocal}
Fix $0 < \cvecmin \leq \cvecgeom \leq \cvecmax$, with $\cvecgeom < 1$. Fix $\mathbf{r} = \cvec(0) \in \cspace{\cvecgeom}{\cvecmin}{\cvecmax}$, define $\cvec(t)$ as in Section~\ref{sec:ctod}, fix $t_0 \in [0, \tau]$, let $\mathbf{s} = \cvec(t_0)$, and let $w \in V \in \eqref{eq:Vset}$ be a root of  
\begin{equation} \label{eq:sformeqn}
P(z; \mathbf{s}) = 1.
\end{equation}
Then there exists a unique continuous function $w(t): \metricball{t_0}{r} \to \CC$, analytic in the interior of $\metricball{t_0}{r}$ with range in $\metricball{t_0}{\kappa}$, where $\kappa$, $r$ depend only on $\cvecgeom$, $\cvecmax$, $\cvecmin$, and $\rho$, such that
\begin{subequations}
\begin{gather}
w(t_0) = w\\
P(w(t); \cvec(t)) = 1 \text{ for all } t \in \metricball{t_0}{r}.
\end{gather}
\end{subequations}
\end{bigclm}

\begin{proof}
To use Appendix~\ref{sec:IFT}, Claim~\ref{clm:IFT} on $F(z; t) \in \eqref{eq:fset}$, we find appropriate estimates for the inequalities \eqref{eq:mOnedef}, \eqref{eq:mTwodef}, \eqref{eq:genOmegadef}.  Note that by $\epsilon + \rho \leq \frac{3}{2} \left( \frac{\cvecmin}{12} \right) < \frac{\cvecmin}{6}$, $V(\rho) \subseteq \extplane(\cvecmin/2)$, and as mentioned above, $t \in J_{\rho}$ implies by Lemma~\ref{lem:extensions} that $\Re c_k(t) \geq \frac{\cvecmin}{2}$ and $\abs{c_k(t)} \leq \sqrt{3} \cvecmax$ for all $k$.

For $M_1$, by $P(z; \cvec(t)) \in \eqref{eq:polyext}$, for all $t \in J_{\rho}$ and $z \in \extplane\left( \frac{\cvecmin}{2} \right)$, $\abs{z} \leq 1 + \rho < 2$,
\begin{equation} \label{eq:partialzunabs}
\begin{split}
\frac{\partial F}{\partial z} = \frac{\partial P}{\partial z} & = \sum_{j = 0}^q m_j (z + d_j(t))^{m_j - 1} \prod_{\substack{k = 1 \\ k \neq j}}^q (z + d_k(t))^{m_k}  = \left[ \sum_{j = 0}^q \frac{m_j}{z + d_j(t)} \right] P(z; \cvec(t))
\end{split}
\end{equation}
and thus by \eqref{eq:cextabsbounds}
\begin{equation}
\begin{split}
\abs{\frac{dP}{dz}} & \leq (q + 1) \left( \sum_{j = 0}^q m_j \right) \times (1 + \rho + \sqrt{3}\cvecmax)^{n - 1}\\
& \leq n^2  2^n( 1 + \cvecmax)^n = n^2 M_0.
\end{split}
\end{equation}
For $\frac{\partial F}{\partial t}$, and for all cases \ref{enum:lowersmaller} -- \ref{enum:qone}, we need only two terms:
\begin{equation} \label{eq:partialtunabs}
\begin{split}
\frac{\partial F}{\partial t} = \frac{\partial}{\partial t}P(z; \cvec(t))  & = \left[  \frac{m_0 \atder(t)}{z + d_0(t)} +  \frac{m_q \btder(t)}{z + d_q(t)} \right] P(z; \cvec(t))\\
& = m_0 m_q \left[ \frac{d_0(t)}{z + d_0(t)} - \frac{d_q(t)}{z + d_q(t)} \right] P(z; \cvec(t))\\
& = - m_0 m_q z (d_q(t) - d_0(t))\cdot \frac{P(z; \cvec(t))}{(z + d_0(t))(z + d_q(t))}.
\end{split} 
\end{equation}
 For $t \in J_{\rho}$, $0 < \rho < \frac{\log 3}{2n}$, we have by \eqref{eq:cextboundingbox} that $\Re c_j(t) > \frac{\cvecmin}{2}$, or $\Re d_j (t) > \frac{\cvecgeom}{2}$, so for all $j$, $0 \leq j \leq q$, and $z \in \extplane( \cvecmin/2 )$,
\begin{equation} \label{eq:resumboundbelow}
\Re(z + d_j) \geq \Re(z + \cvecmin) > - \frac{\cvecmin}{6} + \frac{\cvecmin}{2} = \frac{\cvecmin}{3},
\end{equation}
Using \eqref{eq:resumboundbelow},\eqref{eq:mZerodef}, and $\abs{z} < 2$ in the final line of \eqref{eq:partialtunabs},
\begin{equation} \label{eq:partialtabs}
\begin{split}
\abs{ \frac{\partial F}{\partial t}} & \leq \frac{36 n^2 \cdot M_0}{\cvecmin^2}.
\end{split} 
\end{equation}

Therefore,  we can choose
\begin{equation} \label{eq:mOnefinal}
M_1  =  n^2 M_0 + \frac{36 n^2 M_0}{\cvecmin^2} = n^2 M_0 \left( 1 + \frac{36}{\cvecmin^2}\right).
\end{equation}
As above, we can bound the second derivatives of $F(z; t) \in \eqref{eq:fset}$ and it suffices to choose
\begin{equation} \label{eq:mTwofinal}
M_2 = 216 n^3 (1 + \cvecmax) \left[ 1 + \frac{M_1}{\cvecmin} + \frac{M_0}{\cvecmin^2} \right]
\end{equation}

By \eqref{eq:partialzunabs}, we have that for any particular root $\specw \in \extplane(\cvecmin/2)$ of
\begin{equation} \label{eq:mainequationT}
P(z; \cvec(t)) = 1
\end{equation}
that
\begin{equation} \label{eq:partialsunabs}
\left. \frac{\partial P}{\partial z} \right\vert_{z = \specw} = \sum_{j = 0}^q \frac{m_j}{\specw + d_j(t)},
\end{equation}
so defining
\begin{equation} \label{eq:omegafinal}
\Omega = \frac{n \cvecmin}{6(1 + \cvecmax)^2} ,
\end{equation}
we have that 
\begin{equation} \label{eq:partialsabsbelow}
\Re \left( \left. \frac{dP}{dz} \right\vert_{z = \specw} \right)  = \sum_{j = 0}^q \Re \left(\frac{m_j}{w + d_j} \right) = \sum_{j = 0}^q \frac{m_j \Re (w + d_k)}{\abs{w + d_k}^2} > n \frac{\cvecmin/6}{(1 + \cvecmax)^2}  = \Omega.
\end{equation}

With $\cvec(0) = \mathbf{r}$ and $\cvec(t)$ defined in Section~\ref{sec:ctod}, and $t_0 \in [0, t] = J$, $\mathbf{s} = \cvec(t_0)$ choose $\specw \in V \subseteq \extplane(\cvecmin/2)$ among the roots of  \eqref{eq:sformeqn}.  We choose
\begin{subequations}
\begin{align}
2\kappa &= \min \bracepair*{\rho, \frac{\Omega}{8M_2}} \label{eq:kappafinal}\\
2r & = \min \bracepair*{\kappa \cdot \frac{\Omega}{8(M_1 + M_2)}, \rho} \label{eq:rfinal}. 
\end{align}
\end{subequations}
Then by the Implicit Function Theorem, i.e. by Claim~\ref{clm:IFT}, there exists a continuous function $w(t): \rdisk{r}(t) \to \CC$, analytic in the interior, with image contained on $\metricball{\specw}{k}$, such that
\begin{equation}
P(w(t), \cvec(t)) = 1, \quad w(0) = \specw,
\end{equation}
and with $F(z, t) \in \eqref{eq:fset}$,
\begin{equation}
\wtder(t) = - \, \left. \frac{\partial F}{\partial t} /\frac{\partial F}{\partial z} \right\vert_{z = w(t)}.
\end{equation}
\end{proof}

\begin{bigclm} \label{clm:posder}
In the setting of Claim~\ref{clm:IFTLocal}, whenever $t \in \metricball{t_0}{r}$, $t < \tau$, and $w(t)$ is in the set 
\begin{equation} \label{eq:walldef}
\wall \declare \bracepair{\xi \in \CC: \abs{\Re \xi} \leq \epsilon, \delta \leq \abs{\Im \xi} \leq 1}, \quad \delta \in \eqref{eq:deltadef}, \epsilon \in \eqref{eq:epsilondef},
\end{equation}
we have that
\begin{equation}
\Re \wtder(t) > 0
\end{equation}
\end{bigclm}

\begin{proof}
We now wish to demonstrate that if $t \in (t_0 - r, t_0 + r)$, $t < \tau$, and $\Re w(t) \leq \epsilon$, $\epsilon \in \eqref{eq:epsilondef}$, then $\wtder > 0$.  For real $t$ in this domain, by Lemma~\ref{lem:extensions}, $\cvec(t) \in \cspace{\cvecgeom}{\frac{\cvecmin}{\sqrt{3}}}{\sqrt{3} \cvecmax}$, so when invoking Corollaries~\ref{cor:zerocontainmentannulus}  and Corollary~\ref{cor:zerofreebox}, we will use $d(\cvecgeom, \sqrt{3} \cvecmax)$ and $\delta(\cvecgeom, \sqrt{3}\cvecmax)$.

Consider first the easier case \ref{enum:qone}, i.e., $q = 1$.  The sum \eqref{eq:partialsunabs} has only two terms so
\begin{equation}
\left.\frac{\partial F}{\partial z}\right\vert_{z = w(t)} = \frac{m_0}{w(t) + d_0(t)} + \frac{m_0}{w(t) + d_1(t)} = (m_0 + m_1) \frac{w(t) + \abaverage(t)}{(w(t) + d_0(t))(w(t) + d_1(t))},
\end{equation}
where
\begin{equation}
d_0(t) < \abaverage(t) < d_1(t), \quad \abaverage(t) \declare \frac{1}{m_0 + m_1} ( m_1 d_0(t) + m_0 d_1(t) )
\end{equation}
and since $P(w(t), t) = 1$, we have by \eqref{eq:partialtunabs} that
\begin{equation} \label{eq:partialtPEqualsOne}
\left. \frac{\partial F}{\partial t} \right\vert_{z = w(t)} = - m_0 m_1 \frac{z (d_1(t) - d_0(t))}{(z + d_0(t))(z + d_1(t))}
\end{equation}
Therefore, with $m_0 + m_1 = n$,
\begin{equation} \label{eq:wtderformqone}
\wtder(t) = \frac{m_0 m_1}{n} [ d_1(t) - d_0(t) ] \left( 1 + \frac{\abaverage(t)}{w(t)} \right)^{-1} .
\end{equation} If $w(t) = u(t) + i v(t)$, then
\begin{equation}
\Re \frac{1}{w(t)} = \frac{u(t)}{u(t)^2 + v(t)^2}
\end{equation}
and so 
\begin{equation}
\Re \frac{\abaverage(t)}{w(t)} = \abaverage(t) \cdot \frac{u(t)}{u(t)^2 + v(t)^2}.
\end{equation}
Since $\abs{u(t)} \leq \epsilon < \frac{1}{2} \cdot \frac{d(\cvecgeom, \sqrt{3} \cvecmax)^2}{\sqrt{3} \cvecmax}$, by Lemma~\ref{lem:zerofreedisk},
\begin{equation} \label{eq:absreduce}
\abs{\Re \frac{\abaverage(t)}{w(t)}} \leq \abaverage(t) \cdot \frac{\abs{u(t)}}{d^2} \leq \frac{\sqrt{3} \cvecmax}{d(\cvecgeom, \sqrt{3} \cvecmax)^2} \abs{u(t)} \leq \frac{1}{2}.
\end{equation} 

Then $\Re \left( 1 + \frac{\abaverage(t)}{w(t)} \right) \geq \frac{1}{2} > 0$ 
and by Remark~\ref{rem:inverserealsgn} 
\begin{equation} \label{eq:posderqone}
\Re \wtder(t) > 0 \text{ if } t < \tau.
\end{equation}

In the cases \ref{enum:lowersmaller} -- \ref{enum:luequal}, \eqref{eq:partialtunabs}, with the simplification $P(w(t), t) = 1$,  gives 
\begin{equation} \label{eq:partialtPEqualsOneAgain}
\left. \frac{\partial F}{\partial t} \right\vert_{z = w(t)} = - m_0 m_q \frac{z (d_q(t) - d_0(t))}{(z + d_0(t))(z + d_q(t))}.
\end{equation}
By \eqref{eq:partialsunabs}
\begin{equation} \label{eq:partialzqgen}
\left. \frac{\partial F}{\partial z} \right\vert_{z = w(t)} = \frac{m_0}{w(t) + d_0(t)} + \frac{m_q}{w(t) + d_q(t)} + \sum_{j = 1}^{q - 1} \frac{m_j}{w(t) + d_j},
\end{equation}
and the third term needs special attention, even with constant $d_j$ for $0 < j < q$. By \eqref{eq:partialzqgen} and \eqref{eq:partialtunabs},
\begin{equation} \label{eq:maindiffeqqgen}
\wtder(t) = - \left. \frac{\partial F}{\partial t}/ \frac{\partial F}{\partial z}\right\vert_{z = w(t)} = \frac{m_0 m_q(d_q(t) - d_0(t))}{H(w(t))}
\end{equation}
where
\begin{equation} \label{eq:Hdefqgen}
H(z) \declare  (m_0 + m_q)(1 + \frac{\abaverage(t)}{z}) + \sum_{j = 1}^{q-1} \frac{(z + d_0(t))(z + d_q(t))}{z(z + d_j)}, \quad \abaverage(t) = \frac{m_q d_0(t) + m_0 d_q(t)}{m_0 + m_q}.
\end{equation}
Notice that for $0 < a < c < b$,
\begin{equation} \label{eq:rhodef}
\rho \declare \frac{(z + a)(z + b)}{z(z + c)} = 1 + \frac{a + b -c}{z} + \frac{(c-a)(b-c)}{(-z)(z + c)},
\end{equation}
so 
\begin{equation} \label{eq:Hclarify}
\begin{split}
H(z) &= n + \frac{1}{z} \left( (m_0 + m_q) \abaverage(t) + \sum_{j = 1}^{q - 1} (d_0(t) + d_q(t) - d_j) \right) \\
& \quad + \sum_{j = 1}^{q - 1} \frac{(d_j - d_0(t))(d_q(t) - d_j)}{- z(z + d_j)} \\
& = n + T_2 + T_3.
\end{split}
\end{equation}
The second term $T_2$ in \eqref{eq:Hclarify} --- compare \eqref{eq:absreduce} --- 
\[
\abs{\Re T_2} \leq 2 n \cdot \sqrt{3} \cvecmax  \cdot\frac{\abs{u(t)}}{u(t)^2 + v(t)^2} \leq \frac{2 \sqrt{3} n \cvecmax}{d(\cvecgeom, \sqrt{3}\cvecmax)^2} \cdot \abs{u(t)} \leq \frac{1}{4} n
\]
by
\begin{equation}
\abs{u(t)} \leq \epsilon = \frac{\delta(\cvecgeom, \sqrt{3}\cvecmax)^2}{4(1 + \sqrt{3} \cvecmax)} <  \frac{d(\cvecgeom, \sqrt{3} \cvecmax)^2}{4 \sqrt{3} \cvecmax}
\end{equation}
For the estimates of the sum $T_3$ notice that, with $z = x + iy$, 
\begin{equation} \label{eq:proddetails}
z(z + c) = (x + iy)(x + c + iy) = x(x + c) - y^2 + iy(2x + c)
\end{equation}
and 
\begin{equation} \label{eq:denomdetailsqgen}
\Re \left[ - \, \frac{1}{z(z + c)} \right] = \frac{y^2 - x(x + c)}{(y^2 - x(x + c))^2 + y^2(2x + c)^2}.
\end{equation}
With $c = d_j$, $0 < j < q$, 
\begin{equation}
y^2 - x(x + c) \geq 0 
\end{equation}
if $\abs{y} \geq \delta(\cvecgeom, \sqrt{3} \cvecmax)$ by
\begin{equation}
\delta(\cvecgeom, \sqrt{3} \cvecmax)^2 \geq \epsilon ( \epsilon + 2 \cvecmax), \quad \abs{x} \leq \epsilon.
\end{equation}

Therefore, for $z \in \eqref{eq:walldef}$, $\Re T_3 > 0$; moreover, since $t$ real and in $[-\rho, \tau]$ implies $\cvec(t) \in \cspace{\cvecgeom}{\frac{\cvecmin}{\sqrt{3}}}{\sqrt{3} \cvecmax}$, so by Corollary~\ref{cor:zerofreebox}, $\abs{w(t)} \geq \delta(\cvecgeom, \sqrt{3} \cvecmax)$.   Together with \eqref{eq:denomdetailsqgen} and \eqref{eq:Hclarify}, this implies that
\begin{equation} \label{eq:Hbounds}
\Re H(z) \geq \frac{3}{4} n > 0, \quad z \in \wall.
\end{equation}
and by \eqref{eq:proddetails}, \eqref{eq:rhodef}, and \eqref{eq:maindiffeqqgen}, $\Re \wtder(t) > 0$ if the trajectory $w(t)$ is in the $\wall$, so the root $w(t)$ cannot leave the $\ourbox$ by crossing the $\wall$ to the left 
\begin{figure}
\begin{tikzpicture}[scale = 3.0]
\draw [<->](-1.1, 0) -- (1.1, 0) node[right] {$x = \Re z$};
\draw [<->] (0, -1.1) -- (0, 1.2) node[above] {$y = \Im z $}; 
\draw [very thick] (0, -1) -- (1, -1) -- (1, 1) -- (0, 1) -- (0, 1/8) -- (1/8, 1/8) -- (1/8, -1/8) -- (0, -1/8) -- (0, -1);
\node at (1.2, 1.2) {$\ourbox$};
\draw [->] (1.1, 1.1) -- (0.9, 0.9);
\filldraw[fill = red] (-1/8, 1) -- (1/8, 1) -- (1/8, 1/8) -- (-1/8, 1/8) -- cycle;
\node at (-0.25, 0.5) {$\wall$ $\rightarrow$};
\filldraw[fill = red] (-1/8, -1) -- (1/8, -1) -- (1/8, -1/8) -- (-1/8, -1/8) -- cycle;
\node at (-0.25, -0.5) {$\wall$ $\rightarrow$};
\end{tikzpicture}
\label{fig:argument}
\caption{The $\wall$ of \eqref{eq:walldef} and the $\ourbox$ of \eqref{eq:ourboxdef}}
\end{figure}
(see Figure 2).  
\end{proof}

\begin{bigclm} \label{clm:posextend}
In the setting of Claim~\ref{clm:IFTLocal}, suppose that $t_0 < \tau$ and $\Re \specw \geq 0$.  Then $w(t)$, restricted to $[t_0, t_0 + r]$, extends uniquely to a function on $[t_0, \tau]$ such that 
\begin{equation}
\Re w(t) > 0 \text{ for all } t \in (t_0, \tau].
\end{equation}
\end{bigclm}
\begin{proof}
By Corollary~\ref{cor:zerocontainmentannulus}, any root of $w(t)$, $t$ real, with $\abs{\Re w(t)} < \epsilon$ is in the $\wall$.  By Claim~\ref{clm:posder}, we have that $\Re \wtder(t) > 0$ if $w(t)$ is in the $\wall$ and $t < \tau$.  Thus, whenever $\abs{\Re w(t)} < \epsilon$ and $t < \tau$, $\Re \wtder(t) > 0$.  
\begin{enumerate}[label = \textbf{Case \arabic*.}]
\item If $r \geq \tau - t_0$, for each $\eta > 0$, we may apply Claim~\ref{clm:posFirstForm} with $h(t) = \Re w(t)$, $[a, b] = [t_0, t_0 + \tau - \eta]$, $\smallterm = \epsilon$, so that $\Re w(t) > 0$ for all $t \in (t_0, t_0 + \tau - \eta]$.  Thus, $\Re w(t) > 0$ for all $t \in (t_0, \tau)$.  In addition, $\Re w(\tau) > 0$: if for some interval $(\tau - \eta, \tau)$, $\Re w(t) < \epsilon$ for $t \in (\tau - \eta, \tau)$, then $\Re \wtder(t) > 0$ for $t \in (\tau - \eta, \tau)$, so $\Re w(\tau) > \Re w\left(\tau - \frac{\eta}{2} \right) > 0$.  Otherwise, for all $\eta > 0$, there exists $t \in (\tau - \eta, \tau)$ with $\Re w(t) \geq \epsilon$, so  there exists an increasing sequence $\bracepair{t_j}_{j = 1}^{\infty}$ in $(t_0, \tau)$ with $\Re w(t_j) \geq \epsilon$ for all $j \geq 1$, and
\[
\Re w(\tau) = \lim_{t \nearrow \tau} \Re w(t) = \lim_{j \to \infty} \Re w(t_j) \geq \epsilon.
\]
In all cases, $\Re w(t) > 0$ on $(t_0, \tau]$. 
\item If $r < \tau - t_0$, define
\[
K = \inf \bracepair{k \in \NN: \tau - t_0 \leq k \cdot \frac{r}{2}};
\]
$K \geq 3$ by $r  = \frac{2r}{2} < \tau - t_0$.  Then let 
\begin{equation}
t_k = t_0 + k\frac{r}{2}, \quad 0 \leq k \leq K - 1,
\end{equation} 
and we inductively define $w(t)$ on $\bigcup_{k = 0}^{K - 2} \metricball{t_k}{r}$ as follows. Put $w(t) = w_0(t)$ on $\metricball{t_0}{r}$ as in Claim~\ref{clm:IFTLocal}.  We may apply Claim~\ref{clm:posFirstForm} with $h(t) = \Re w_0(t)$, $[a, b] = [t_0, t_1]$, $\smallterm = \epsilon$, to ensure $\Re w_0(t) > 0$ on $\left(t_0, t_1\right]$.

Suppose that we have defined $w(t)$ on $\bigcup_{k = 0}^j \metricball{t_k}{r}$, $j \leq K - 3$, and ensured that $\Re w(t) > 0$ on $(t_0, t_{j + 1}]$; we now show how to extend the definition to $\bigcup_{k = 0}^{j + 1} \metricball{t_k}{r}$ and ensure positive real part on $(t_0, t_{j + 2}]$.  Since $\Re w(t_{j + 1}) > 0$ by hypothesis, we may define $w_{j + 1}(t)$ on $\metricball{t_{j + 1}}{r}$ by Claim~\ref{clm:IFTLocal}, the unique function such that $w_{j + 1}(t_{j + 1}) = w(t_{j + 1})$ and $P(w_{j + 1}(t), \cvec(t)) = 1$ for all $t \in \metricball{t_{j + 1}}{r}$.  We have $w(t_{j + 1}) = w_{j + 1}(t_{j + 1}) \in \metricball{t_j}{r} \cap \metricball{t_{j + 1}}{r}$, so by the uniqueness statement for $w_{j + 1}$, $w_{j + 1}(t) = w_j(t)$ for all $t$ in $\metricball{t_j}{r} \cap \metricball{t_{j + 1}}{r}$. We extend the definition of $w(t)$ by 
\[
w(t) = \begin{cases} w_{\text{old}}(t), & t \in \bigcup_{k = 0}^j \metricball{t_k}{r}\\
w_{j + 1}(t),& t \in \metricball{t_{j + 1}}{r} \end{cases},
\]
which is a valid definition by the equality on the overlap.  Moreover, $t_{j + 2} \in \overset{\circ}{\mathbb{D}_{r}}(t_{j + 1})$, and $j \leq K - 3$, so $j + 2 \leq K - 1$, so $t_{j + 2} \leq t_{K - 1} < \tau$ by definition of $K$, and so we may apply Claim~\ref{clm:posFirstForm} to $h(t) = \Re w(t)$, $[a, b] = [t_0, t_{j + 2}]$, $\smallterm = \epsilon$ to demonstrate that $\Re w(t) > 0$ on $(t_0, t_{j + 2}]$.

By induction, we have defined $w(t)$ on $\bigcup_{k = 0}^{K - 2} \metricball{t_k}{r}$, such that $\Re w(t) > 0$ on $(t_0, t_{K - 1}]$.  As in our induction argument, we may expand the definition of $w(t)$ to include $\metricball{t_{K - 1}}{r}$, but now $\tau - t_{K - 1} \leq t_0 + K(r/2) - [t_0 + (K - 1)(r/2)] = r/2 < r$, so as in Case 1, we may argue that $\Re w(t) > 0$ on $[t_{k - 1}, \tau]$, hence on $(t_0, \tau]$, in this last step.  
\end{enumerate}
\end{proof}
\begin{smallrem} \label{rem:smallLeft}
In Claim~\ref{clm:posextend} , we could replace ``$\Re \specw \geq 0$'' by ``$\Re \specw \geq - \epsilon$'' for the starting point and ``$\Re w(t) > 0$'' by ``$\Re w(t) > - \epsilon$'' for $t > t_0$ -- for in the invocations of Claim~\ref{clm:posFirstForm}, we could have taken $h(t) = \Re w(t) + \epsilon$ and $\smallterm = 2 \epsilon$, since Claim~\ref{clm:posder} ensures that $\wtder(t) > 0$, hence $h^{\prime}(t) > 0$, if $\abs{\Re w(t)} \leq \epsilon$, i.e. $0 \leq h(t) \leq 2 \epsilon$.  Thus, we can start a little to the left of the imaginary axis and have a path on $[t_0, \tau]$.  
\end{smallrem}

\noindent \textit{Completion of the proof of Proposition~\ref{prop:omnibus}.}  The primary step remaining is to show an appropriate choice of $\eta$ such that I can extend to all $t \in J_{\eta}$.  By Claim~\ref{clm:posextend}, if $\Re \specw \geq 0$, we have a nice function $w(t)$ on $[0, \tau]$ with $w(0) = \specw$ and $\Re w(t) > 0$ for $t \in (0, \tau]$.  At each point $t \in [0, \tau]$, we have an $r$-radius ball where the function $w(t)$ can be extended, and the uniqueness from the Implicit Function Theorem ensures that these extensions are consistent.  Therefore, $w(t)$ exists on $J(r)$ according to the model of \eqref{eq:rhonbhddef}.  Setting $\eta = \frac{2r}{3}$, $J_{\eta} \in \eqref{eq:rhonbhddef}$ is a subset of $J(r) \in \eqref{eq:Jrhodef}$, by the same reasoning as in the proof of Corollary~\ref{cor:zerofreebox}.

Since $\eta = \frac{2r}{3} \leq \frac{\rho}{3} < \frac{\log 3}{2n}$, Lemma~\ref{lem:extensions} ensures that for $t \in [-\eta, \tau + \eta]$, $\cvec(t) \in \cspace{\cvecgeom}{\frac{\cvecmin}{2}}{2\cvecmax}$, and hence by Corollary~\ref{cor:zerocontainmentannulus} all nonnegative roots are in $\ann{d(\cvecgeom, 2\cvecmax)}{1} \subset \ann{d(\cvecgeom, \cvecmax) /2}{1}$.  The result on the sign of the derivative follows from Claim~\ref{clm:posder}.
\end{proof}

\section{Movement of the Roots} \label{sec:movement}

\begin{bigclm} \label{clm:firstmonotonicity}
$\countproots{\cvec(t)}$ and $\countpzroots{\cvec(t)}$ are nondecreasing on $[0, \tau]$
\end{bigclm}
\begin{proof}
First, we prove the inequality on $[0, \tau]$,
\begin{equation} \label{eq:partialineqbelow}
\countproots{\cvec(0)} \leq \countproots{\cvec(t)}, \quad \countproots{\cvec(0)} \leq \countproots{\cvec(t)}
\end{equation}

So far, we have talked about the trajectory $w(t)$ of \emph{one} root $w(0) = \specw$.  All roots in the $\ourbox$ are simple by Lemma~\ref{lem:simplezeros} so at no instant $t$ do two of the $\countproots{\cvec(0)}$ trajectories  with the initial $\countproots{\cvec(0)}$ root-points could coalesce; yet they remain in the $\ourbox$ and $E^+$ (or $\overline{E}$).  New roots could come from the left, i.e, from $E^- = \bracepair{\xi \in \CC: \Re \xi < 0}$, but this only pushes up the number $\countproots{\cvec(t)}$ in the right half-plane so $\countproots{(c(0)} \leq \countproots{\cvec(t)}$, $0 \leq t \leq \tau$.  The same can be said about roots in $\overline{E}$ so $\countpzroots{\cvec(0)} \leq \countpzroots{\cvec(t)}$.

We will get the full claim if we show 
\begin{equation} \label{eq:above}
\countproots{\cvec(t^{\prime})} \leq \countproots{\cvec(t)} \text{ and } \countpzroots{\cvec(t^{\prime})} \leq \countpzroots{\cvec(t)}, \quad 0 \leq t^{\prime} <  t \leq \tau.
\end{equation}
Without changing the structure or \diver\ $q$ of the multiset \eqref{eq:dlist} -- \eqref{eq:mlist} let us only change $d_0$ to $\widetilde{d_0} = d_0 \exp (m_q t^{\prime})$ and $d_q$ to $\widetilde{d_q} = d_q \exp( - m_0 t^{\prime})$.  With $t^{\prime} < \tau$ the inequalities 
\begin{equation}
\begin{split}
\widetilde{d_0} < d_1 \, \text{ and } \, d_{ q - 1} < \widetilde{d_q} &\quad \text{in Cases \ref{enum:lowersmaller} -- \ref{enum:luequal}},\\
\text{ or } \widetilde{d_0} < \widetilde{d_1} &\quad \text{in Case \ref{enum:qone}},
\end{split}
\end{equation} 
will be preserved.  If we proceed by the scheme of Section~\ref{sec:ctod} with the initial multiset 
\[
\widetilde{d_0} < d_1 < \dotsb < d_{q-1} < \widetilde{d_q} \quad (or \widetilde{d_0} < \widetilde{d_1}),
\]
and the old multiplicities $\bracepair{m_j}_{j = 0}^q$, recalculation of $\tau$ leads to 
\[
\begin{split}
\widetilde{\tau_*} = \frac{1}{m_q} \log \left( \frac{d_1}{d_0 \exp(m_q t^{\prime})} \right) = \tau_* - t^{\prime},\\
\widetilde{\tau^*} = \frac{1}{m_q} \log \left( \frac{d_q \exp(-m_0 t^{\prime})}{d_{q - 1}}  \right) = \tau_* - t^{\prime},\\
\text{ and } \widetilde{\tau} = \tau - t^{\prime}, 
\end{split}
\]
in Cases \ref{enum:lowersmaller} -- \ref{enum:luequal}, or
\[
\widetilde{\tau} = \frac{1}{m_0 + m_1} \log \left( \frac{d_1 \exp(-m_0 t^{\prime})}{d_0 \exp(m_1 t^{\prime})} \right) = \tau - t^{\prime}
\]
in Case \ref{enum:qone}. 

Now we can apply the work of Section~\ref{sec:method}, with the understanding that $[t^{\prime}, \tau]$ is shifted by $t^{\prime}$ to $[0, \widetilde{\tau}]$, and \eqref{eq:partialineqbelow} becomes the inequalities \eqref{eq:above}.
\end{proof}

In Section~\ref{sec:ctod} we made one step $\bracepair{\cvec(t) \in \RR^n, \quad 0 \leq t \leq \tau}$, or $\bracepair{d(t) \in \RR^{q + 1}; 0 \leq t \leq \tau}$, to bring the \diver\  $q$ down by $1$ or $2$, with numbers of zeroes $\countproots{\cvec}$, $\countpzroots{\cvec}$ of $P(z; \cvec) - 1$ in $E^+$ and $E$ not decreasing.

We can repeat the same construction (many times, but at most $q$ times) if $q^{\prime} > 0$ still, treating the end--multiset of the previous set as the initial sequence~\eqref{eq:cbounds}, or multiset \eqref{eq:dlist} for the next step.  In this way, we get the intervals $[\tau_i, \tau_{i + 1}]$, $\tau_0 = 0$, $\Delta_i = \tau_{i + 1} - \tau_i > 0$,  $i = 0, 1, \dotsc, p-1$; $p \leq q$, and the following holds.
\begin{prop} \label{prop:finale}
Fix $\cvec_0 \in (\rplus)^n$, with geometric mean $\cvecgeom$.  There exists a continuous, piecewise--$C^{\infty}$ function
\[
\cvec(t): [0, T] \to (\rplus)^n, \quad T = \sum_{i = 0}^{p-1} \Delta_i, \, \, p \leq q, \, \, \Delta_i > 0,
\]
such that
\begin{enumerate}[label = (\alph*)]
\item $\cvec(0) = \cvec_0 \in \eqref{eq:cbounds}$
\item $\cvec(T) = \cvecideal = (\cvecgeom, \cvecgeom, \dotsc, \cvecgeom)$.
\item \label{enum:everincr}$\countproots{\cvec(t)}$ and $\countpzroots{\cvec(t)}$ are non-decreasing functions on $[0, T]$, i.e., for all $t$, $t^{\prime}$, $0 \leq t^{\prime} < t \leq T$,
\[
\countproots{\cvec(t^{\prime})} \leq \countproots{\cvec(t)} \text{ and } \countpzroots{\cvec(t^{\prime})} \leq \countpzroots{\cvec(t)}.
\]
\end{enumerate}
\end{prop}
\begin{proof}
We explained these claims in Section~\ref{sec:method}. 
\end{proof}
The $t = 0$, $t^{\prime} = T$ case of \ref{enum:everincr} is precisely Theorem~\ref{thm:maincount}.

We now describe more precisely the movement of the roots.  Define for $t \in [0, T]$

\begin{equation} \label{eq:omegatdef}
\begin{split}
\countnrootstime{t} & = \countnroots{\cvec(t)}\\
\countzrootstime{t} & = \countzroots{\cvec(t)}\\
\countprootstime{t} & = \countproots{\cvec(t)}\\
\countpzrootstime{t} & = \countpzroots{\cvec(t)}
\end{split}
\end{equation}
\begin{bigclm} \label{clm:countjumps}
The counting function $\countproots{t}$ has a point of discontinuity at $t = t^*$ if and only if
\begin{equation} \label{eq:mainLateRevise}
P(z; c(t^*)) - 1  = 0
\end{equation}
has roots on $i\RR$. 
\end{bigclm}
\begin{proof}
If such roots do not exist, then define $h > 0$ by $2h =  \min \bracepair*{ \frac{\cvecmin}{3}, \min_{\substack{z \in \CC\\ P(z, \cvec(t^*)) = 1 }} \bracepair*{\abs{ \Re z}}}$.  Define the region
\begin{equation} \label{eq:ghdef}
G_h = \bracepair*{z \in \CC: \Re z \geq h, \abs{z} \leq 1},
\end{equation}
and note that by the Cauchy Integral Formula, e.g., \cite[Section 4.7, pp. 97 -- 99]{Conway},
\begin{equation}
\countprootstime{t^*} = \frac{1}{2\pi i} \int\limits_{\bd{G_h}} \frac{dP/dz}{P(z; \cvec(t^*)) - 1} \, dz.
\end{equation}
Let
\begin{equation} \label{eq:mudef}
\min\limits_{z \in \bd{G(h)}} \abs{P(z; \cvec(t) - 1} = \mu(t),
\end{equation}
then $\mu(t^*) > 0$, and $\mu(t)$ is continuous at $t^*$, so there exists $\rho > 0$, $\rho < 1$ such that
\begin{equation} \label{eq:muposrules}
\abs{\mu(t) - \mu(t^*)} \leq \frac{1}{2} \mu(t^*) \quad \text{if} \quad \abs{t - t^*} \leq \rho.
\end{equation}
Thus, $\mu(t) > 0$ for $t \in [t^* - \rho, t^* + \rho]$, and the Cauchy integral
\begin{equation}
\eta_+(t) = \frac{1}{2\pi i} \int\limits_{\bd{G_h}} \frac{dP/dz}{P(z; \cvec(t)) - 1} \, dz.
\end{equation}
is continuous on  $[t^* - \rho, t^* + \rho]$, but it is integer-valued, being the counting-function for the roots of 
\begin{equation} \label{eq:mainLateReviseGen}
P(z; \cvec(t)) - 1 = 0
\end{equation} 
in the interior of $G_h$, so $\eta_+(t)$ is constant.  Therefore, the number of roots of \eqref{eq:mainLateReviseGen} in $G_h$ is $\omega(t^*)$ for all $t$ in $[t^* - \rho, t^* + \rho]$.

To show that $\eta_+(t) = \countprootstime{t}$, we must show that no roots enter from the left.  We know that there are no roots to \eqref{eq:mainLateRevise} in the strip $\bracepair{ \xi \in \CC: \abs{\Re \xi} < 2 h}$, in particular on the imaginary axis, so we consider
\[
G_{0} = \bracepair*{z \in \CC: \Re z \geq 0, \abs{z} \leq 1}.
\]
In the same way, shrinking $\rho$ if necessary,  for $t \in [t^* - \rho, t^* + \rho]$ there are no roots of \eqref{eq:mainLateRevise} on $\bd{G_0}$ (in particular, on the imaginary axis), and on $[t^* - \rho, t^* + \rho]$, the function 
\[
\overline{\eta}(t) =  \frac{1}{2\pi i} \int\limits_{\bd{G_{0}}} \frac{dP/dz}{P(z; \cvec(t)) - 1} \, dz
\]
is a continuous counting-function, hence constant.  $\overline{\eta}(t^*) = \eta^+(t^*) = \countprootstime{t^*}$, as there are no roots at time $t^*$ with small real part, so as $\overline{\eta}$ and $\eta_+$ are constant on this interval, and there are no roots on the imaginary axis, the number of roots in $G_{0} \setminus G_h$ is $0$ for all $t \in [t^* - \rho, t^* + \rho]$.  All roots in the closed right-half-plane must be in $G_0$ by Lemma~\ref{lem:zerofreedisk}, so we must have $\countprootstime{t} = \eta^+(t)$ for $t \in [t^* - \rho, t^* + \rho]$.  Thus, $\countprootstime{t}$ is constant for $t \in [t^* - \rho, t^* + \rho]$ for some small $\rho$.  (Since we showed there were no roots on the imaginary axis, $\countprootstime{t} = \countpzrootstime{t}$ for $t \in [t^* - \rho, t^* + \rho]$, so the counting-function on the closed half-plane is also constant).

Suppose that \eqref{eq:mainLateRevise} has a pure imaginary root; since solutions to \eqref{eq:maineqn} imply solutions to \eqref{eq:maineqnabstwo}, we have by Lemmas~\ref{lem:oddcount} and \ref{lem:zerofreedisk} that the root can only be $\pm i Y$, $1 > Y > d > 0$, $d \in \eqref{eq:ddef}$, only two roots, and they are simple by Lemma~\ref{lem:simplezeros}.  Now $t^* \in [0, T]$ can be one of three types of points: 
\begin{enumerate}[label = (\roman*)]
\item \label{enum:tinterior} $t^* \in (t_k, t_{k + 1})$ for some $k$, $0 \leq k < p \leq q - 1$;
\item \label{enum:tboundary} $t^* = t_k$, $0 \leq k < p$;
\item \label{enum:tlast} $t^* = T$.
\end{enumerate} 
We need to know well the behaviour of the root $\widetilde{w}(t)$ with $\widetilde{w}(t^*) = i Y$, for $t$ around $t^*$, a well-determined function for small $\rho$ by the Implicit Function Theorem. 

The case \ref{enum:tinterior} is easier: with $\rho$, $0 < \rho < \frac{1}{2} \min \bracepair{t_{k + 1} - t^*, t^* - t_k}$ the path $\cvec(t)$, or $\mathbf{d}(t)$, is defined by formulas \eqref{eq:dkdefhgen}, \eqref{eq:cdinvert}; this is analytic on $I_{\rho} = [t^* - \rho, t^* + \rho]$, or even if we talk about complex $t$ in a neighborhood of $I_{\rho} \subseteq \CC$.  By \eqref{eq:posderqone}, $K =  \left. \Re \wttder(t) \right\vert_{t = t^*} > 0$, so for $\rho$ small enough,
\begin{equation} \label{eq:wttderlowerbd}
\Re \wttder(t) \geq \frac{1}{2} K > 0, \quad t \in I_{\rho},
\end{equation}
so
\begin{subequations} 
\begin{align}
\Re \widetilde{w}(t) &= \Re (\widetilde{w}(t) - \widetilde{w}(t^*)) = \Re \int_{t^*}^t \wttder(\sigma) \, d\sigma \geq \frac{1}{2} K (t - t^*), && \text{ if } t > t*, \label{eq:wtrightform}\\
\Re \widetilde{w}(t) & - \Re (\widetilde{w}(t^*) - \widetilde{w}(t)) = - \Re \int_{t}^{t^*} \wttder(\sigma) \, d\sigma \leq - \frac{1}{2} K(t^* - t), && \text{ if } t < t^*.
\end{align}
\end{subequations}
Now we know the past and the future of the roots $\pm i Y(\cvec(t^*))$: they are in $E^+$ if $t^* \leq t \leq t^* + \rho$, and they are in $E^-$ if $t^* - \rho \leq t < t^*$.   All other roots remain in their half-planes; it can be explained as in \eqref{eq:ghdef} -- \eqref{eq:muposrules} (with $G_{-h}$ in the place of $G_0$).  Therefore,
\begin{equation} \label{eq:countpjump}
\countprootstime{t} = \begin{cases} \countprootstime{t^*} + 2, & t^* < t \leq t^* + \rho\\
\countprootstime{t}, & t^* - \rho \leq t \leq t^* \end{cases}
\end{equation}
and
\begin{equation} \label{eq:countpzjump}
\countpzrootstime{t} = \begin{cases} \countpzrootstime{t^*} = \countprootstime{t^*} + 2, & t^* \leq t \leq t^* + \rho\\
\countpzrootstime{t*} - 2 = \countprootstime{t^*}, & t^* - \rho \leq t < t^* \end{cases}
\end{equation}
The case~\ref{enum:tboundary}, i.e, $t^* = t_k$, $0 \leq k \leq p - 1$, is more delicate because the function $\cvec(t)$, or $\mathbf{d}(t)$, at $t^*$ is only continuous.  It is defined by different $C^{\infty}$ (or analytic) functions on $[t_{k - 1}, t_k]$ and $[t_k, t_{k + 1}]$.  Thus, we need to analyze them and their derivatives on $[t_{k - 1}, t_k)$ and $(t_k, t_{k + 1}]$ separately.  Claim~\ref{clm:posder} gives us all the information; the later case $(t_k, t_{k + 1}]$ is an analogue of $(0, \tau]$, so 
\begin{equation} \label{eq:wttderleftend}
\Re \wttder(t_k) > 0
\end{equation}
and we can repeat \eqref{eq:wttderlowerbd} and \eqref{eq:wtrightform} to justify the claim
\begin{equation}
\Re \widetilde{w}(t) > 0 \text{ if } t_k < t \leq t_k + \rho.
\end{equation}
(This also suffices for the case $t^* = 0$, i.e., $k = 0$).  If, however, $t < t_k$, the derivative $\widetilde{w}(t)$ by \eqref{eq:wtderformqone} or \eqref{eq:maindiffeqqgen} has two factors 
\begin{equation} \label{eq:wttdersplit}
\wttder(t) = F(\widetilde{w}(t)) \cdot \Delta(t), 
\end{equation}
and by \eqref{eq:absreduce}, \eqref{eq:Hbounds},
\begin{equation} \label{eq:Fboundslow}
\Re F(\widetilde{w}(t)) \geq K_* > 0, \quad \abs{t - t^*} \leq \rho,
\end{equation}
and $\Delta(t) = d_q(t) - d_0(t)$, with $\Delta(t) \geq \Delta_* > 0$ for some $\Delta_*$ in the cases \ref{enum:lowersmaller}, \ref{enum:uppersmaller}, \ref{enum:luequal} (a), and $\left. \Delta(t) \right\vert_{t = t^*} = 0$ in the cases \ref{enum:luequal} (b), \ref{enum:qone}.  Yet $\Delta(t) = d_q \exp (- m_0 t) - d_0\exp(m_q t)$, and we have $\left. \frac{d \Delta}{dt} \right\vert_{t = t^*} = - (m_0 d_q + m_q d_0)d_0 e^{m_q t^*}$,
so 
\begin{equation} \label{eq:Deltaderboundslow}
\frac{d \Delta}{dt}(t) \leq - L_* (t^* - t), \quad \text{ for some } L_* > 0 \text{ and } \rho \ll 1.
\end{equation}
Therefore, for $t = t^* - h$, $0 \leq h \leq \rho$, by \eqref{eq:Fboundslow} and \eqref{eq:Deltaderboundslow}, 
\begin{equation}
\Re \widetilde{w}(t) = - \Re [\widetilde{w}(t^*) - \widetilde{w}(t)] = - \int_{t^* - h}^{t^*} \Re \wttder(\xi) \, d\xi \geq K_* L_* \int_0^h \eta \, d\eta = \frac{h^2}{2} K_* L^* > 0. 
\end{equation}
Therefore, as in Case~\ref{enum:tinterior}, the inequalities~\eqref{eq:wttderleftend} and \eqref{eq:Deltaderboundslow} justify \eqref{eq:countpjump} and \eqref{eq:countpzjump} if $t^* = \tau_k$, $0 \leq k < p$.

The case~\ref{enum:tlast} is special; it happens only if $\cvec(t) = c^*$, i.e., $[\tau_{p - 1}, T]$ is an analogue of $[0, \tau]$ in Cases~\ref{enum:luequal}(b) or \ref{enum:qone}.  As in \eqref{eq:wttdersplit}, a pure imaginary root comes from the left, so
\begin{subequations}
\begin{align}
\countprootstime{t} = \countprootstime{T} & \text{ if } T - \rho \leq t \leq T\\
\countpzrootstime{t} = \countpzrootstime{T} - 2 & \text{ if } T - \rho \leq t < T.
\end{align}
\end{subequations}
\end{proof}
 
\section{Construction of multisets such that \texorpdfstring{$\countproots{\cvec} =1 $ and $\countpzroots{\cvec} = 1$}{nu+(c) = 1 and nubar(c) = 1} } \label{sec:nuOne}

We have demonstrated that for all $\cvec \in (\rplus)^n$ with geometric mean $\cvecgeom < 1$, the maximum values for $\countproots{\cvec}$ and $\countpzroots{\cvec}$ are achieved by $\cvecideal$.  The question is what the minimum value is, or could be.

As per Lemma~\ref{lem:oddcount}, if the geometric mean $\cvecgeom < 1$, then $\countproots{\cvec} \geq 1$ and $\countpzroots{\cvec} \geq 1$ by the positive real root.  We now show that this lower bound is the minimum.

\begin{prop} \label{prop:countoneachieved}
Fix $n \in \NN$,  and fix $\cvecgeom \in (0, 1)$.  There exists $\cvec \in (\rplus)^n$ with geometric mean $\cvecgeom$ such that $\countproots{\cvec} = \countpzroots{\cvec} = 1$.  
\end{prop} 

We note that for $1 \leq n \leq 4$, $0 < \cvecgeom < 1$,
\[
(z + \cvecgeom)^n = 1
\]
has only one root with positive real part, as follows from \eqref{eq:specialsolform}, so $\countproots{\cvecideal} = \countpzroots{\cvecideal} = 1$.  In the sequel, we assume that $n \geq 5$.

It turns out that control of the $2$ extreme coordinates in $\cvec$ suffices to force the number of eigenvalues in the right-half-plane to be equal to $1$.  For convenience, let $n^{\prime} = n - 2$. 
\begin{prop} \label{prop:countoneachievers}
Let $n \in \NN$, $n \geq 5$, and fix $\cvec^{\prime} \in (\rplus)^{n^{\prime}}$.  Then for any $\cvecgeom \in (0, 1)$, there exists a vector $\cvecext = (d_0, \cvec^{\prime}, d_q) \in (\rplus)^n$ with geometric mean $\cvecgeom$ and $\countproots{\cvecext} = \countpzroots{\cvec_{\text{ext}}} = 1$.   
\end{prop}

To begin the proof, we write $\cvec^{\prime}$ in $d$-notation as $\dlong{\mathbf{d}^{\prime}}{m^{\prime}}{q^{\prime}}$.  We extend $\cvec^{\prime}$ to $\cvecext$ by setting $q = q^{\prime} + 2$ and choosing $d_0 < d^{\prime}_0$ and $d_q > d^{\prime}_{q^{\prime}}$ such that $d_0 \left[\prod_{j = 0}^{q^{\prime}} \left(d^{\prime}_j\right)^{m^{\prime}_j} \right] d_q = \cvecgeom^n$.  Set
\[
d_j = \begin{cases} d_0, & j = 0\\
d^{\prime}_{j - 1}, & 1 \leq j \leq q - 1 = q^{\prime} + 1,\\
d_q, & j = q \end{cases}
\]
\[
m_j = \begin{cases} 1, & j = 0\\
m^{\prime}_{j - 1}, & 1 \leq j \leq q - 1 = q^{\prime} + 1,\\
1, & j = q \end{cases}
\]
and let $\cvecext$ be the resulting vector in $(\rplus)^n$ as created by \eqref{eq:cdinvert}.  Altogether, setting $P(z; \cvecext) = 1$, we have
\begin{equation} \label{eq:laterestart}
\begin{split}
(z + d_0) \left( \prod_{j = 1}^{q - 1} (z + d_j)^{m_j} \right)(z + d_q) = 1;\\
d_0 \left( \prod_{j = 1}^{q - 1} d_j^{m_j} \right) d_q = \cvecgeom^n , \quad n^{\prime} = \sum_{j = 1}^{q-1} m_j, 
\end{split}
\end{equation}
We rescale the coefficients to make $\cvecgeom$ evident:
\begin{equation}
d_0 = \frac{\cvecgeom}{A}, \quad d_{q} = M \cvecgeom; \quad d_j = b_j \cvecgeom,\quad 1 \leq j \leq q - 1,
\end{equation}
and \eqref{eq:laterestart} becomes
\begin{equation}
\left( z + \frac{\cvecgeom}{A} \right)  \left[ \prod_{j = 1}^{q - 1} (z + b_j \cvecgeom)^{m_j} \right] (z + M \cvecgeom ) = 1.
\end{equation}
We rescale $z$ as $z = \cvecgeom w$ to move all the $\cvecgeom$ terms to the other side, which does not change the signs of the real parts of any zeroes; letting $G = \frac{1}{\cvecgeom}$ we have
\begin{subequations}
\begin{align}
\left( w + \frac{1}{A} \right) \left[  \prod_{j = 1}^{q - 1} (w + b_j )^{m_j} \right] (w + M ) = G^n, \label{eq:dscaleeqn}\\
\frac{1}{A} B^{n^{\prime}} M = 1, \quad B^{n^{\prime}} = \prod_{j = 1}^{q - 1} b_j^{m_j}. \label{eq:dscaleform}
\end{align}
\end{subequations}
We will fix $d_j$, $1 \leq j \leq q - 1$, i.e., the $(b_j)_{j = 1}^{q-1}$ of \eqref{eq:dscaleform}, and in \eqref{eq:dscaleeqn} 
\begin{equation} \label{eq:blist}
b_0 = \frac{1}{A} < b_1 < \dotsb < b_{q - 1} < b_q = M.
\end{equation}
and (thinking of $M$ as our parameter, and $A$ varying as in \eqref{eq:dscaleform} to balance the geometric mean) study the polynomial
\begin{equation}
P_M \equiv P(w; \mathbf{b}, M) \text{ of the left-hand-side of } \eqref{eq:dscaleeqn}.  
\end{equation}
and the roots of \eqref{eq:dscaleeqn}. Similarly to the previous, we define $\mu_+(M)$ and $\overline{\mu}(M)$ to be the number of roots of $P_M(z) = G^n$ in the open right-half-plane $E^+$ and the closed half-plane $E$, respectively .
\begin{bigclm} \label{clm:rescaleOne}
If $M \in \eqref{eq:blist}$ is sufficiently large, and $A = M \cdot B^{n^{\prime}}$ as in \eqref{eq:dscaleform}, then
\begin{equation} \label{eq:isone}
 \mu_+(M) = \overline{\mu}(M) = 1.
\end{equation}
\end{bigclm}
\begin{proof}
As in Lemma~\ref{lem:oddcount}, there is a guaranteed root in $(0, G)$ by the Intermediate Value Theorem, as by \eqref{eq:dscaleform},
\[
P_M(0) = \frac{1}{A} B^{n^{\prime}} M = 1 < G^n < \left( G + \frac{1}{A} \right) \left[ \prod_{j = 1}^{q - 1} (G + b_j)^{m_j} \right] (G + M) = P_M(G).
\]
Thus, $\overline{\mu}(M) \geq \mu_+(M) \geq 1$.

Put 
\begin{equation} \label{eq:betadef}
2 \beta = \min \bracepair*{ \frac{b_1}{2}, \min \bracepair{b_{j + 1} - b_j: 1 \leq j \leq  q - 2}};
\end{equation}
then the interiors of the disks 
\begin{equation}
\mathbb{D}_j = \bracepair{w \in \CC: \abs{w + b_j} \leq \beta}, \quad 1 \leq j \leq q - 1, 
\end{equation}
do not intersect, and their closures are in $E^-\in \eqref{eq:Esets}$.  If, in addition, $M \geq \frac{2}{B^n b_1}$, then 
\[b_1 - b_0 = b_1 - \frac{1}{A} \geq b_1 - \frac{b_1}{2} = \frac{b_1}{2} \geq 2 \beta\]
and the disks $\mathbb{D}_0$ and $\mathbb{D}_1$ have disjoint interior.  Similarly, if $M \geq b_{q - 1} + 2 \beta$, then the interiors of $\mathbb{D}_{q - 1}$ and $\mathbb{D}_q$ do not intersect, and $\mathbb{D}_q$ is contained in the open left-half plane.  We wish to show that for $1 \leq j \leq q$,
\begin{equation} \label{eq:boundcirclebound}
\min \bracepair{\abs{P_M(w)}: \abs{w + b_j} = \beta} = \min\limits_{w \in  \bd{\mathbb{D}_{j}}} \bracepair{\abs{P_m(w)} }\geq 2 G^n
\end{equation}
if $M \geq \max  \bracepair*{\frac{2}{\beta B^{n^{\prime}}}, 2(b_{q - 1} + \beta), \frac{8}{\beta^{n- 1}} G^{n + 2}}$.  Indeed, if $k \neq j$, $0 \leq j, k \leq q$,
\[
\abs{w + b_k} = \abs{(w + b_j) + (b_k - b_j)} \geq \abs{b_k - b_j} - \abs{w + b_j} \geq 2 \beta - \beta = \beta,
\]
so for $w \in \bd{\mathbb{D}_j}$, $1 \leq j \leq q$,
\[
\begin{split}
\abs{P_M(w)} &\geq \abs{w + \frac{1}{A} } \beta^{n^{\prime}} \abs{w + M} \\
& \geq \left( \beta - \frac{1}{A} \right) \beta^{n^{\prime}} (M - b_{q - 1} - \beta)\\
& \geq \frac{1}{2} \beta \cdot \beta^{n^{\prime}} \cdot \frac{1}{2} M \geq \frac{\beta^{n - 1}}{4} M
\end{split}
\]
if $\frac{1}{A} \leq \frac{1}{2} \beta$, or $M \geq \frac{2}{\beta B^{n^{\prime}}}$, and $M \geq 2(\beta + b_{q - 1})$.  We choose \\
$M \geq \max \bracepair*{\frac{2}{\beta B^{n^{\prime}}}, 2(b_{q - 1} + \beta), \frac{8}{\beta^{n- 1}} G^{n + 2}}$.  Then \eqref{eq:boundcirclebound} holds.

The first term does not exceed the third (because $b_j \geq 2j \beta$, $j \geq 1$, by \eqref{eq:betadef}), so we choose
\begin{equation} \label{eq:mchoice}
M = 8 \max \bracepair*{b_{q - 1},  \frac{G^n}{\beta^{n - 1}}}.
\end{equation}

Then the number of roots of $P_M(w) = \xi$ does not depend on $\xi$ if $\xi \leq \frac{1}{2} \min\limits_{w \in  \bd{\mathbb{D}_{j}}} \bracepair{\abs{P_M(w)} }$, as this is the integral 
\[
N(\mathbb{D}_j) = \frac{1}{2\pi i} \int\limits_{\bd{\mathbb{D}_j}} \frac{\disp \frac{dP_M}{dw}}{P_m(w) - \xi} \, dw \quad \text{ (with the counterclockwise orientation)}
\]
(see, e.g., \cite[Section 4.7, pp. 97 -- 99]{Conway}).  In particular, when $\xi = 0$, this is $m_j$, $1 \leq j \leq q$ (this includes the case $j = q$, i.e., $b_q = M$).  Therefore, counted with multiplicity, the number of roots of $P_m(w) = G^n$ is also $m_j$ for $M$ as above.  This holds for $j = 1, 2, \quad q$.  For such $j$, all such $\mathbb{D}_j$ are in the left half-plane, as $\beta \leq  b_1 - b_0 < b_1$.  Hence, there are $\sum_{j = 1}^{q} m_j = n^{\prime} + 1 = n - 1$ roots in the left half-plane.  Since we have already located the zero in the right-half plane, we have accounted for all roots of the $n$th-degree polynomial $P_M(w) - G^n$: $n - 1$ in the open left-half-plane, none on the imaginary axis, and $1$ in the right half-plane. Thus, $\mu_+(M) = \overline{\mu}(M) = 1$ for $M$ as in \eqref{eq:mchoice}.  
\end{proof}

\section{The Range of \texorpdfstring{$\countproots{\cvec}$ and $\countpzroots{\cvec}$}{nu+(c) and nubar(c)}}  \label{sec:nuAll}
For each $(n, \cvecgeom)$ pair, we have shown the maximum and minimum values for $\countproots{\cvec}$ and $\countpzroots{\cvec}$ for all $\cvec \in (\rplus)^n$ with geometric mean $\cvecgeom$.  By Lemma~\ref{lem:oddcount}, $\countproots{\cvec}$ and $\countpzroots{\cvec}$, are odd integers, but we wish to show that every odd number between the minimum values and maximum values of $\countproots{\cvec}$ and $\countpzroots{\cvec}$ is achieved.

For convenience, combining Lemma~\ref{lem:oddcount} and \eqref{eq:specialsolcountform}, write
\begin{subequations} \label{eq:kappaforms}
\begin{align}
\countproots{\cvecideal} &= 2 \kappa_+ + 1\\
\countpzroots{\cvecideal} &= 2 \overline{\kappa} + 1.
\end{align}
\end{subequations}

\begin{cor} \label{cor:nuInterp}
Fix $n \in \NN$, and $\cvecgeom \in (0, 1)$, and fix $\cvec_0 \in (\rplus)^n$ with geometric mean $\cvecgeom$.  Let $\cvecmin = \min_j c_j$ and $\cvecmax = \max_j c_j$.  Construct $\cvec(t): [0, T] \to \cspace{\cvecgeom}{\cvecmin}{\cvecmax}$ as defined in Sections~\ref{sec:ctod} and \ref{sec:movement}. Then: 
\begin{enumerate}[label = (\roman*)]
\item for any odd $k$, $\countproots{\cvec} \leq k \leq 2 \kappa_+ + 1$, there exists $t = t(k) \in [0, T]$ with  $\countprootstime{t(k)} = k$.
\item for any odd $\ell$, $\countpzroots{\cvec} \leq \ell \leq 2 \overline{\kappa} + 1$, there exists $\overline{t}(\ell) \in [0, T]$ with $\countpzrootstime{\overline{t}(\ell)} = k$.
\end{enumerate}
\end{cor}

\begin{proof} 
By Claim~\ref{clm:countjumps}, or its proof, the jumps of $\countprootstime{t}$ and $\countpzrootstime{t}$ are of size $2$, and the points of discontnuity $t^*$ are where the equation~\eqref{eq:mainLateRevise} has pure imaginary roots.  There are $\mu_+ = \frac{1}{2} \left[(2 \cvecgeom + 1) - \countproots{\cvec} \right]$ point of discontinuity, by Lemma~\ref{lem:oddcount}, named $\bracepair*{\eta_j}_{j = 1}^{\mu_+}$, and 
\begin{equation}
\countprootstime{t} = \omega_j = 2 p_j + 1, \quad \eta_j < t \leq \eta_{j + 1}, \quad p_{j + 1} = p_j + 1, \text{ or } 0 < \tau \leq \eta_1.  
\end{equation}
\begin{equation}
\countpzrootstime{t} = \omega_j, \quad \eta_j \leq t < \eta_{j + 1}.  
\end{equation}
These facts on the structure of the functions $\countprootstime{t}, \countpzrootstime{t}$ imply (i), (ii).
\end{proof}

Since by Proposition~\ref{prop:countoneachieved}, we know for all $(n, \cvecgeom)$ pairs with $n \in \NN$, $\cvecgeom \in (0, 1)$, there exists $\cvec_0 \in (\rplus)^n$ with geometric mean $\cvecgeom$ and $\countproots{\cvec_0} = \countpzroots{\cvec_0} = 1$, we may apply Corollary~\ref{cor:nuInterp} to such a $\cvec_0$ and achieve all positive odd values less than the maximum.  This proves the following.  

\begin{prop} \label{prop:allvals}
Fix $n \in \NN$, and $\cvecgeom \in (0, 1)$.  Then:
\begin{enumerate}[label = (\roman*)]
\item for all odd $k$, $1 \leq k \leq 2 \kappa_+ + 1$, there exists $\cvec \in (\rplus)^n$ with  $\countproots{\cvec} = k$.
\item for all odd $\ell$, $1 \leq \ell \leq 2 \overline{\kappa} + 1$, there exists $\cvec \in (\rplus)^n$ with $\countpzroots{\cvec} = \ell$.
\end{enumerate}
\end{prop}

In the context of doubly cyclic matrices, we have the following.

\begin{prop} \label{prop:allvalsMatrices}
Fix $n \in \NN$, and $0 < \alpha < \beta < 1$.  Then:
\begin{enumerate}[label = (\roman*)]
\item for all odd $k$, $1 \leq k \leq 2 \kappa_+ + 1$, there exists $X \in DC(\alpha, \beta)$ with $k$ roots in the open left half-plane  with  $\countproots{\cvec} = k$.
\item  for all odd $\ell$, $1 \leq \ell \leq 2 \overline{\kappa} + 1$, there exists $X \in DC(\alpha, \beta)$ with $\ell$ roots in the closed left half-plane.
\end{enumerate}
\end{prop}

Since by Lemma~\ref{lem:oddcount} and Theorem~\ref{thm:maincount}, $\countproots{\cvec}$ (respectively, $\countpzroots{\cvec}$) is odd and less than $\countproots{\cvecideal}$ (respectively, $\countpzroots{\cvecideal}$), the range is no larger than that demonstrated in Proposition~\ref{prop:allvals}, so this completes the proof of the $\cvecgeom < 1$ case of Theorem~\ref{thm:polyCountRangeFull}; the $\cvecgeom \geq 1$ case was handled at the beginning of Section~\ref{sec:tech}.  Similarly, we have proven Theorem~\ref{thm:matrixCountRangeFull}

\section{Further Comments} \label{sec:furtherOne}
In the proof of the main theorem we used the localization
\begin{equation} \label{eq:basicannboundsredo}
w \in \ann{d}{1}, \quad d = \frac{1 - \cvecgeom^n}{(1 + \cvecmax)^n},
\end{equation}
of roots of \eqref{eq:maineqn} in the right half-plane $E$, or in $\ourbox \in \eqref{eq:ourboxdef}$.  We want now to improve the upper bound $1$ in \eqref{eq:basicannboundsredo} and make explicit the dependence on $\cvecgeom = \left( \prod_{k = 1}^n c_j \right)^{1/n}$.  As in \eqref{eq:absabovebasic},
\begin{equation} \label{eq:absaboverev} 
\begin{split}
1 &= \prod_{j = 1}^n \left[ (x^2 + y^2) + 2x c_j + c_j^2\right] \\
& > \prod_{j = 1}^n (\abs{z}^2 + c_j^2) = \abs{z}^{2n} \prod_{j = 1}^n \left( 1 + \left(\frac{c_j}{\abs{z}}\right)^2 \right)\\
& \geq \abs{z}^{2n} \left(1 + \frac{\cvecgeom^2}{\abs{z}^2}\right)^n \geq \left( \abs{z}^2 + \cvecgeom^2\right)^n ,
\end{split}
\end{equation}
so 
\begin{equation}
1 - \cvecgeom^2 > \abs{z}^2,
\end{equation}
and by $\cvecgeom < 1$,
\begin{equation} \label{eq:UBimproved}
\abs{z}^2 \leq 2 (1 - \cvecgeom), \quad \abs{z} \leq \frac{3}{2} (1 - \cvecgeom)^{1/2}.
\end{equation}

The last inequality in \eqref{eq:absaboverev} follows from 
\begin{equation} \label{eq:geommeanboundbelow}
\prod_{k = 1}^n (1 + t_k) \geq (1 + T)^n, \quad T^n = \prod_{k = 1}^n t_k, \quad \text{if } t_k \geq 0, \quad 1 \leq k \leq n,
\end{equation}
with strict inequality unless all $t_k$ are equal.  See \cite[{\#}64, p. 61]{HLP52}.

The lower bound $d \in \eqref{eq:basicannboundsredo}$ can be improved also.  Notice that
\begin{equation}
2ab \leq \omega^2 a^2 + \Omega^2 b^2, \quad \omega \Omega = 1, \quad a, b \geq 0, \quad 1 \geq \omega > 0,
\end{equation}
so for all $j$
\begin{equation}
2 x c_j \leq \Omega^2 x^2 + \omega^2 c_j^2,
\end{equation}
and
\begin{equation} \label{eq:lateexpansion}
\begin{split}
1 & = \prod_{j = 1}^n \left[ (x^2 + y^2) + 2x c_j + c_j^2\right] \\
 & \leq \prod_{j = 1}^n \left( y^2 + x^2 ( 1 + \Omega^2 ) + c_j^2 (1 + \omega^2) \right) = (1 + \omega^2)^n \cvecgeom^{2n} \prod_{j = 1}^n \left[ 1 + \frac{y^2 + (1 + \Omega^2) x^2}{1 + \omega^2} \frac{1}{c_j^2} \right].
 \end{split}
\end{equation}
Notice that 
\begin{equation} \label{eq:logbounds}
\log(1 + u) < u, \, \, \text{ for all } \, \, u > 0, \quad \text{ and }\log(1 + u) \geq \frac{3}{4} u, \, \, \text{ if } \, \, 0 < u \leq \frac{1}{3}.  
\end{equation}
Then taking logarithms of both sides in \eqref{eq:lateexpansion},
\begin{equation}
0 \leq n \left[ \log ( (1 + \omega^2) \cvecgeom^2) + \frac{1}{n} \sum_{j = 1^n} \log \left( 1 + \frac{y^2 + (1 + \Omega^2)x^2}{1 + \omega^2} \frac{1}{c_j^2} \right) \right].
\end{equation}
and
\begin{equation} \label{eq:lbstarterkit}
\log \frac{1}{\gamma^2(1 + \omega^2)} \leq \frac{y^2 + (1 + \Omega^2)x^2}{(1 + \omega^2)} L^2; \quad L^2 = \frac{1}{n} \left( \sum_{j = 1}^n \frac{1}{c_j^2} \right).
\end{equation}
With $\cvecgeom^2 < 1$ choose $\omega > 0$ by
\begin{equation}
(1 + \omega^2)\cvecgeom^2 = \frac{1}{2} \left( 1 + \cvecgeom^2 \right);
\end{equation}
then 
\begin{equation}
\frac{1 + \Omega^2}{1 + \omega^2} = \frac{1}{\omega^2} = \frac{2\gamma^2}{1 - \cvecgeom^2}= \frac{2\gamma^2}{1 + \gamma} \cdot \frac{1}{1 - \gamma} < \frac{1}{1- \cvecgeom},
\end{equation}
and \eqref{eq:lbstarterkit} becomes
\begin{equation} \label{eq:lbcontinued}
\log \frac{2}{(1 + \cvecgeom^2)} \leq \left( y^2 + \frac{x^2}{1 - \cvecgeom} \right) L^2, \quad \text{ where } L^2  \leq \frac{1}{\cvecmin^2}.
\end{equation}
But $\frac{2}{1 + \cvecgeom^2} = 1 + \frac{1 - \cvecgeom^2}{1 + \cvecgeom^2}$, and if $\cvecgeom \geq \frac{1}{2}$, \, by \eqref{eq:logbounds}, 
\begin{equation}
\log \frac{2}{1 + \cvecgeom^2} \geq \frac{3}{4} \cdot \frac{1 - \cvecgeom^2}{1 + \cvecgeom^2} \geq \frac{3}{4}  (1 - \cvecgeom)^2,
\end{equation}
and by \eqref{eq:lbcontinued},
\begin{equation}
\frac{3}{4}(1 - \cvecgeom) \leq \left[ y^2 + \frac{x^2}{1 - \cvecgeom} \right] L^2,
\end{equation}
and
\begin{equation}
 \frac{3}{4}\cvecmin^2 \leq \frac{x^2}{(1 - \cvecgeom)^2} + \frac{y^2}{1 - \cvecgeom}. 
\end{equation}
It implies with $0 < \cvecgeom < 1$ that
\begin{equation}
 \frac{3}{4}\cvecmin^2 \leq \frac{x^2 + y^2}{(1 - \cvecgeom)^2}.
\end{equation}
Therefore (in conjunction with \eqref{eq:UBimproved}), as an analogue or an improvement of Lemmas~\ref{lem:basicUB}, \ref{lem:zerofreedisk} and Corollary~\ref{cor:zerocontainmentannulus}, we can state the following.
\begin{bigclm} \label{clm:boundsimprovement}
If $\cvec \in \cspace{\cvecgeom}{\cvecmin}{\cvecmax}$, $1 > \cvecgeom \geq \frac{1}{2}$, and $z$ is a root of \eqref{eq:maineqn} in the right half-plane, then
\begin{enumerate}[label = (\roman*)]
\item $z$ lies outside of the ellipsoid 
\begin{equation}
\frac{x^2}{(1 - \cvecgeom)^2} + \frac{y^2}{1 - \cvecgeom} < \frac{3}{4} \cvecmin^2,
\end{equation}
or
\item in a weaker claim,
\begin{equation}
\frac{4}{5} \cvecmin (1 - \cvecgeom) \leq \abs{z}.
\end{equation}
\end{enumerate}
Therefore $z \in \ann{ \frac{4}{5} \cvecmin (1 - \cvecgeom)}{\frac{3}{2} (1 - \cvecgeom)^{1/2}}$.
\end{bigclm}

A slight advantage over Corollary~\ref{cor:zerocontainmentannulus} and Lemma~\ref{lem:zerofreedisk} is that there is no $n$ in Claim~\ref{clm:boundsimprovement}, at least explicitly.  Speaking loosely, we can say that the area of localization changes continuously when $\cvecgeom$ goes from $\cvecgeom > 1$ to $\cvecgeom < 1$.

\appendix

\section{Proof of Inequalities \eqref{eq:geommeanboundbelow}} \label{sec:furthertwo}  To make our paper self-contained, we will explain the inequality \eqref{eq:geommeanboundbelow}.  

Let us consider the elementary symmetric polynomials 
\begin{equation} \label{eq:sympolydef}
\sigma_{k}(\mathbf{x}) = \sum_{\substack{K \subset \setn \\ |K| = k}} X(K); \quad X(K) =  \prod_{j \in K} x_j, \quad \mathbf{x} \in \RR^n,\quad \setn \in \eqref{eq:klist}
\end{equation}
and the averages
\begin{equation} \label{eq:symavgdef}
S_k(\mathbf{x}) = \frac{\sigma_k(\mathbf{x})}{\binom{n}{k}} , \quad \binom{n}{k} = \# \bracepair*{K \subset \setn : |K| = k}.
\end{equation}
C. Maclaurin, in 1729, in \cite{Macl}, proved a chain of inequalities in the case $\mathbf{x} \in (\rplus)^n$: 
\begin{equation} \label{eq:MacIneq}
S_1(\mathbf{x}) \geq S_2^{1/2}(\mathbf{x}) \geq \dotsc \geq S_k^{1/k}(\mathbf{x}) \geq \dotsc \geq S_n^{1/n}(\mathbf{x}) = g(\mathbf{x}) \declare \left(\prod_{k = 1}^n x_k \right)^{1/n},
\end{equation}
with strict inequality (at least once) if and only $x_j \neq x_k$ for some $j, k \in \setn$, $j \neq k$.  To explain \eqref{eq:geommeanboundbelow} we need just the individual inequalities
\begin{equation} \label{eq:Macsimple}
S_k^{1/k}(\mathbf{x}) \geq g(\mathbf{x}), \quad 1 \leq k < n.
\end{equation}
For $\mathbf{x} \in (\rplus)^n$,
\begin{equation}
\begin{split}
\prod_{j = 1}^n (1 + x_j) &= \sum_{k = 0}^n \sigma_k(\mathbf{x})  = \sum_{k = 0}^n \binom{n}{k} S_k(\mathbf{x}) \\
&\geq \sum_{k = 0}^n \binom{n}{k} [g(\mathbf{x})]^k = (1 + g(\mathbf{x}))^n.
\end{split}
\end{equation}
To prove \eqref{eq:Macsimple}, let us notice that by \eqref{eq:symavgdef}, $S_k$ is an arithmetic mean of $\binom{n}{k}$ positive numbers $X(K)$, but their product is a homogeneous polynomial 
\[
\prod_K X(K) = \left( \prod_{j = 1}^n x_j \right)^G,
\]
of degree $nG = k \binom{n}{k}$, so the AM-GM inequality implies
\begin{equation}
S_k(\mathbf{x}) \geq \left[ \left([g(\mathbf{x})]^n\right)^G \right]^{\binom{n}{k}^{-1}} = g(\mathbf{x})^k.
\end{equation}
We have proven \eqref{eq:Macsimple}.

See more on the Newton and Maclaurin Inequalities in \cite{NicuNewt}, \cite{NicuInterp}, and references therein.  

\section{Implicit Function Theorem} \label{sec:IFT}

Of course, the Implicit Function Theorem is well-known (see, e.g., \cite[Thm. 9.28, pp. 224]{BabyRudin} or \cite[Thm. 7.6, p. 34]{FrGr}), but we use a version with explicit lower bounds on the neighborhoods of validity, so we give the full details below.  For a convex, closed, bounded set $V \subseteq \CC$, put for $0 < \rho < 1$ the $\rho$-neighborhood of $V$,
\begin{equation} \label{eq:rhonbhddef}
V(\rho) = \bracepair{z \in \CC: \abs{z - v} \leq \rho \text{ for some }v \in V}.
\end{equation}
Let $F(z, t)$ be an analytic function of two variables in 
\begin{equation} \label{eq:dualnbhddef}
\mathcal{G} = V(\rho) \times J(\rho), \quad J = [a,b] \in \RR,
\end{equation}
where $J(\rho)$ is the $\CC$-neighborhood of $J$ as in \eqref{eq:rhonbhddef}.  Assume that
\begin{equation} \label{eq:zerosetdef}
Z = \bracepair{(z, t) \in V \times J: F(z, t) = 0}
\end{equation}
is not empty, and
\begin{equation} \label{eq:genOmegadef}
0 < \Omega = \min \bracepair*{ \abs{\frac{\partial F}{\partial z}(z, t)}: (z, t) \in Z}.
\end{equation}
Put
\begin{subequations}
\begin{align}
M_1 &= \max \bracepair*{\abs{\frac{\partial F}{\partial z}(z, t)} + \abs{\frac{\partial F}{\partial t}(z, t)}: (z, t) \in \mathcal{G}} \label{eq:mOnedef}\\
\intertext{and}
M_2 & = \max \bracepair*{\abs{\frac{\partial^2 F}{\partial z^2}(z, t)} + \abs{\frac{\partial^2 F}{\partial z \partial t}(z, t)} + \abs{\frac{\partial^2 F}{\partial t^2}(z, t)}: (z, t) \in \mathcal{G}} \label{eq:mTwodef}
\end{align}
\end{subequations}

Now, choose and fix $\kappa$, $r$ such that
\begin{equation}
0 < \kappa \leq \min \bracepair*{\rho, \frac{\Omega}{8M_2}} \label{eq:kappadef}
\end{equation}
and
\begin{equation} \label{eq:rdef}
0 < r \leq \min \bracepair*{\kappa \cdot \frac{\Omega}{8(M_1 + M_2)}, \rho}
\end{equation}

\begin{bigclm} \label{clm:IFT}
Under the assumptions and notation \eqref{eq:rhonbhddef} -- \eqref{eq:rdef}, if
\begin{equation} \label{eq:startzero}
* = (z_0, t_0) \in V \times J \quad \text{and} \quad F(z_0, t_0) = 0
\end{equation}
then there exists a unique continuous function $z(t)$ in the closed disc
\begin{equation} \label{eq:tSolnnbhd}
\rdisk{r}(t_0) = \bracepair{t \in \CC: \abs{t - t_0} \leq r},
\end{equation}
analytic in the open disk $\overset{\circ}{\rdisk{r}}(t_0)$, such that for $t \in \rdisk{r}$,
\begin{subequations}
\begin{align}
\abs{z(t) - z_0} \leq \kappa,\label{eq:zSolnnbhd}\\
F(z(t), t) = 0. \label{eq:eqsolnhere} 
\end{align}
\end{subequations}
Moreover, for $t \in \overset{\circ}{\rdisk{r}}(t_0)$
\begin{equation} \label{eq:fracder}
\dot{z}(t) = - \, \frac{\partial F}{\partial t} /\frac{\partial F}{\partial z} 
\end{equation}
\end{bigclm}

\begin{proof}
With
\begin{equation} \label{eq:firstdersshorthand}
A = \left. \frac{\partial F}{\partial z} \right\vert_*, \quad B = \left. \frac{\partial F}{\partial t} \right\vert_*
\end{equation}
by the Taylor formula and \eqref{eq:startzero}
\begin{equation} \label{eq:TaylorDecomp}
F(z, t) = A(z - z_0) + B(t - t_0) + g(z, t), \quad g \text{ analytic in }V(\rho) \times J(\rho).
\end{equation}
Put
\begin{equation} \label{eq:errorform}
z = z_0 + \zeta, \quad t = t_0 + s;
\end{equation}
we want(see \eqref{eq:zSolnnbhd}) to find $\zeta(s) \in X$, where $X$ is the Banach space $(\mathcal{A}(\rdisk{r}(0)), \supnorm{\cdot})$ of functions analytic in the open disk $\overset{\circ}{\rdisk{r}}(0)$ and continuous on the closed disk $\rdisk{r}(0)$, such that
\begin{equation} \label{eq:idealTaylor}
0 = A \zeta(s) + Bs + G(\zeta(s), s), \quad s \in \rdisk{r},
\end{equation}
where
\begin{equation} \label{eq:remainderLocal}
G(\zeta, s) = F(z_0 + \zeta, t_0 + s) - A \zeta - B s, \zeta \in \rdisk{\rho}, \, s \in \rdisk{\rho}.
\end{equation}
Define in $X$ the mapping
\begin{equation} \label{eq:phidef}
\Phi: \xi(\cdot) \mapsto  - \frac{B}{A} s - \frac{1}{A} G(\xi(s), s),
\end{equation}
at least for functions with $\abs{\xi(s)} \leq \rho$ for all $s \in \rdisk{r}(0)$.  
The ball
\begin{equation} \label{eq:Knbhddef}
K(\kappa) = \bracepair{\xi \in \mathcal{A}(\rdisk{r}(0)): \abs{\xi(s)} \leq \kappa \text{ if } \abs{s} \leq r}
\end{equation}
is invariant under $\Phi$: indeed, by \eqref{eq:kappadef}, \eqref{eq:rdef}, if $\abs{s} \leq r \leq 1$, 
\begin{equation}
\begin{split}
\abs{\Phi[\xi](s)} &\leq \frac{\abs{B}}{\abs{A}} r + \frac{1}{\abs{A}} M_2(\kappa + r)^2 \leq \frac{1}{\Omega} \left[ M_1 r + 2 M_2 r^2 + 2 M_2 \kappa^2 \right]\\
& \leq \frac{1}{\Omega} \left[ r(M_1 + 2 M_2) + (2 M_2 \kappa) \cdot \kappa \right] \leq \left( \frac{1}{4} + \frac{1}{4} \right) \kappa = \frac{1}{2} \kappa.
\end{split}
\end{equation}
Moreover, on $K(\kappa)$ this mapping is contractive in the uniform norm: If $\xi(s)$, $\zeta(s) \in K(\kappa)$, then
\begin{equation}
\begin{split}
\Phi[\xi](s) - \Phi[\zeta](s) &= \frac{1}{A} \left( G(\zeta(s), s) - G(\xi(s), s) \right)\\
& =  \frac{1}{A}\left( F(z_0 + \zeta, t_0 + s) - A \zeta \right) - \left( F( z_0 + \zeta, t_0 + s) - A \xi \right)\\
& = \frac{1}{A} \int_0^1 \frac{d}{du} \left[ F(z_0 + \xi + u(\zeta - \xi), t_0 + s) - A(\xi + u(\zeta - \xi)) \right] \, du\\
& = \frac{L}{A}(\zeta - \xi),
\end{split}
\end{equation}
where
\begin{equation}
L = \int_0^1  \left[ \frac{\partial F}{\partial z}(z_0 + \xi + u(\zeta - \xi), t_0 + s) - \left. \frac{\partial F}{\partial z} \right\vert_* \right], 
\end{equation}
and $\abs{L} \leq M_2 \kappa$, so
\begin{equation}
\abs{\Phi[\xi](s) - \Phi[\zeta](s)} \leq \frac{M_2 \kappa}{\Omega} \supnorm{\zeta - \xi} \leq \frac{1}{2} \supnorm{\zeta - \xi};
\end{equation}
i.e., $\Phi$ is contractive on $(K(\kappa), \supnorm{\cdot})$.

By the Contractive Mapping Principle, we have a solution $\zeta(s)$ of the equation \eqref{eq:idealTaylor}, or the solution $z(t)$ of \eqref{eq:eqsolnhere}, $z(t_0) = t_0$, and \eqref{eq:zSolnnbhd}.  The solution of \eqref{eq:idealTaylor} is unique in $K(\kappa)$.

The form of the derivative \eqref{eq:fracder} follows from implicit differentiation. 
\end{proof}

\section{Condition for Positivity of Function} \label{sec:poscond}

Fix $a < b$ and let $h \in C^2[a, b]$ be a real-valued function.  Suppose that there exists a positive constant $\smallterm$ such that
\begin{subequations} \label{eq:startingstarterconds}
\begin{gather}
0 \leq h(a) \label{eq:startref}\\
\text{If }0 \leq h(c) \leq \smallterm, \text{ then } h^{\prime}(c) > 0. \label{eq:derrange}
\end{gather}
\end{subequations}

\begin{bigclm} \label{clm:posFirstForm}
If $h \in C^2[a, b]$ and $\smallterm > 0$ satisfies \eqref{eq:startingstarterconds}, then $h(x) > 0$ for all $x$ in $(a, b]$.
\end{bigclm}

\begin{proof}
If $h(a) < \smallterm$, we have by \eqref{eq:derrange} that $h^{\prime}(a) > 0$, so we have $\smallterm \geq h(x) > h(a)$ if $a \leq x \leq a + \rho$, $0 < \rho \ll 1$.  Define
\[
\omega^* = \sup \bracepair{\omega \in (a, b]: \text{ if } x \in (a, \omega), \text{ then } 0 < h(x) < \smallterm}
\]
$\omega^* \geq a + \rho > a$ by the above.  Since $h^{\prime}(x) > 0$ on $(a, \omega)$, $h(x) > h(a) \geq 0$ for all $x$ in $(a, \omega]$, and we are done; also, either $\omega^* = b$, or $\omega^* < b$ and $h(\omega^*) = \smallterm$.  In the former case, $h(x) > h(a) $ for all $x$ in $(a, b]$.  In the latter case, i.e., $\omega^* < b$, we claim that 
\begin{equation} \label{eq:hminright}
h(x) \geq \frac{\smallterm}{2} \text{ if }x \in (\omega^*, b].
\end{equation} 
Otherwise, for some $y \in (\omega^*, b]$, 
\[
h(y) \leq \frac{\smallterm}{2}.
\]
Then the set
\[
T = \bracepair*{t \in [\omega^*, y]: h(t) = \frac{3}{4} \smallterm}
\]
is not empty, and closed; therefore, it contains $t^* = \sup T$, and $\omega^* < t^* < y$.  For $t$, $t^* < t < y$, $h(t) \leq \frac{3}{4} \smallterm = h(t^*)$, so $h^{\prime}(t) >0 $ if $t^* < t < y$.  In particular, letting $k = h^{\prime}(t^*) >0 $, by $h \in C^1$ there exists a small interval $[t^*, t^{\dagger}]$ with $h^{\prime} \geq \frac{k}{2}$ on $[t^*, t^{\dagger}]$.  Therefore, 
\[
0 > - \, \frac{1}{4} \smallterm \geq h(y) - h(t^*) = \int_{t^*}^y h^{\prime}(u) \, du > \int_{t^*}^{t^{\dagger}} k \, dt + \int_{t^{\dagger}}^{y} 0 \, dt \geq k(t^{\dagger} - t^*) > 0.
\]
This contradiction shows that no such $y$ can exist, so $h(x) \geq \frac{\smallterm}{2}$ for all $x \in [\omega^*, b]$.

If $h(a) \geq \smallterm$, we set $\omega^* = a$, and we see that $h(x) \geq \frac{\smallterm}{2}$ for all $x$ in $(a, b] = (\omega^*, b]$ as in the above proof.
\end{proof}
\begin{smallrem}
With a slight adjustment of the proof, we may also replace \eqref{eq:hminright} with the stronger inequality $h(x) \geq \smallterm$.
\end{smallrem}

\originalsectionstyle

\bibliographystyle{alpha}
\bibliography{behaviors.arXiv.refs}
\end{document}